\documentclass[leqno]{amsart}
\usepackage{amsfonts,amssymb,amsmath,amsgen,amsthm}
\usepackage{mathtools}
\newcommand{\msc}[2][2000]{%
  \let\@oldtitle\@title%
  \gdef\@title{\@oldtitle\footnotetext{#1 \emph{Mathematics subject
        classification.} #2}}%
}

\usepackage{xcolor}
\usepackage{tikz,tkz-tab}

\usepackage{envmath}
\usepackage{enumitem}
\usepackage{url}

\usepackage[
breaklinks=true,
colorlinks=false,
bookmarks=true,bookmarksopenlevel=2]{hyperref}

\theoremstyle{plain}
\newtheorem{theorem}{Theorem}[section]

\newtheorem{lem}[theorem]{Lemma}
\newtheorem{cor}[theorem]{Corollary}
\newtheorem{definition}[theorem]{Definition}

\theoremstyle{remark}
\newtheorem{rem}[theorem]{Remark}

\numberwithin{equation}{section}

\renewcommand{\Re}{\operatorname{Re}}

\renewcommand{\Im}{\operatorname{Im}}

\newcommand*\diff{\mathop{}\!\mathrm{d}}

\newcommand{\wsconv}{\overset{\ast}{\rightharpoonup}}

\DeclarePairedDelimiter\abs{\lvert}{\rvert}%
\DeclarePairedDelimiter\norm{\lVert}{\rVert}%
\newcommand{\tnorm}[1]{{\left\vert\kern-0.25ex\left\vert\kern-0.25ex\left\vert #1 
    \right\vert\kern-0.25ex\right\vert\kern-0.25ex\right\vert}}

\makeatletter
\let\oldabs\abs
\def\abs{\@ifstar{\oldabs}{\oldabs*}}
\let\oldnorm\norm
\def\norm{\@ifstar{\oldnorm}{\oldnorm*}}
\makeatother

\def\C{{\mathbb C}}
\def\R{{\mathbb R}}

\def\({\left(}
\def\){\right)}
\def\<{\left\langle}
\def\>{\right\rangle}

\def\Eq#1#2{\mathop{\sim}\limits_{#1\rightarrow#2}}

\def\d{{\partial}}
\def\eps{\varepsilon}

\numberwithin{equation}{section}

\begin{document}

\title[Logarithmic Gross-Pitaevskii equation]{Logarithmic
  Gross-Pitaevskii equation}  
\author[R. Carles]{R\'emi Carles}
\address{Univ Rennes, CNRS\\ IRMAR - UMR 6625\\ F-35000
  Rennes, France}
\email{Remi.Carles@math.cnrs.fr}

\author[G. Ferriere]{Guillaume Ferriere}
\address{Univ. Lille, Inria, CNRS, UMR 8524\\ Laboratoire Paul
  Painlev\'e\\ F-59000 Lille\\ France}
\email{guillaume.ferriere@inria.fr}
\begin{abstract}
 We consider the Schr\"odinger equation with a logarithmic nonlinearty
 and non-trivial boundary conditions at infinity. We prove that the
 Cauchy problem is globally well posed in the energy space, which
 turns out to correspond to the energy space for the standard
 Gross-Pitaevskii equation with a cubic nonlinearity, in small
 dimensions.  We then characterize the solitary and
 traveling waves in the one dimensional case. 
\end{abstract}

\thanks{RC is supported by Centre Henri Lebesgue,
  program ANR-11-LABX-0020-0. A CC-BY public copyright license has been applied by the authors to the present document and will be applied to all subsequent versions up to the Author Accepted Manuscript arising from this submission.}

\maketitle


\section{Introduction}
\label{sec:intro}
We consider the Schr\"odinger equation with a logarithmic nonlinearity,
\begin{equation} \tag{logGP}
    i  \partial_t u + \Delta u = \lambda u \ln{\abs{u}^2}, \quad
   u_{\mid t=0} =u_0, \label{logGP}
\end{equation}
where $x \in \mathbb{R}^d$, $d \geq 1$, $\lambda > 0$, and with the
boundary condition  at infinity
\begin{equation*}
    \abs{u (t, x)} \rightarrow 1 \quad \text{as} \quad \abs{x} \rightarrow \infty.
\end{equation*}
Such boundary condition is reminiscent of the standard
Gross-Pitaevskii equation
\begin{equation}
  \label{eq:GP}
  i  \partial_t u + \Delta u =\(|u|^2-1\)u,\quad
   u_{\mid t=0} =u_0,
\end{equation}
whose Cauchy problem was studied in
\cite{BeSa99,GalloGP,Gerard_Cauchy_GP,PG08,Zhidkov92,Zhidkov01}. The
most complete result regarding this aspect is found in
\cite{Gerard_Cauchy_GP,PG08}, where, for $d\le 3$, \eqref{eq:GP} is proved to be
globally well posed in the energy space
\begin{equation*}
  E_\textnormal{GP}\coloneqq \{ u \in H^1_\textnormal{loc} (\mathbb{R}^d) \, | \, \mathcal{E}_\textnormal{GP} (u) < \infty \},
\end{equation*}
where $\mathcal{E}_\textnormal{GP}$ is the Ginzburg-Landau energy
\begin{equation*}
   \mathcal{E}_{\textnormal{GP}} (u) \coloneqq \norm{\nabla
     u}_{L^2(\R^d)}^2 + \frac{1}{2}\int_{\R^d} \(|u|^2-1\)^2\diff x.
\end{equation*}
See also \cite{KiMuVi16} for an analogous energy-critical problem. 
The logarithmic nonlinearity was introduced in the context of
Schr\"odinger equations in \cite{BiMy76}, as it is the only
nonlinearity satisfying the following tensorization property: if
$u_1(t,x_1),\dots ,u_d(t,x_d)$ are solutions to the one-dimensional
equation, then $u(t,x):=u_1(t,x_1)\times\dots\times u_d(t,x_d)$ solves
the $d$-dimensional equation. This model has regained interest in
various domains of physics: quantum mechanics \cite{yasue},
quantum optics \cite{BiMy76,hansson,KEB00,buljan}, nuclear
physics \cite{Hef85}, Bohmian mechanics \cite{DMFGL03}, effective
quantum gravity \cite{Zlo10}, theory of 
superfluidity and  Bose-Einstein condensation \cite{BEC}.
The logarithmic model may generalize the Gross-Pitaevskii equation, used in
the case of two-body interaction, to the case of three-body
interaction; see \cite{Zlo10,Zlo11}.

\subsection{Cauchy problem}
\label{sec:cauchy}

From the mathematical point of view, the Cauchy problem for
\eqref{logGP} is more intricate than it may seem at first
glance. The difficulty does not lie in the behavior of the logarithm
at infinity, of course, but in its singularity at the origin. In the
case of vanishing boundary condition at infinity,
$u_0\in H^1(\R^d)$, the Cauchy problem was solved in
\cite{cazenave-haraux} (case $\lambda<0$, see also
\cite{HayashiM2018}) and \cite{carlesgallagher,GuLoNi10} (for any
$\lambda\in \R$), by constructing solutions of a regularized equation
converging to a solution to the exact equation, which turns out to be
unique. Even when nontrivial boundary conditions at infinity are
considered, like in the present paper, the singularity of the
logarithm at the origin is the main difficulty, as the cancellation of
the wave function is difficult (if not impossible) to rule out for
\emph{all} $(t,x)\in \R\times \R^d$.  The nonlinearity fails to be
locally Lipschitz continuous, so the fixed point argument employed in
\cite{Gerard_Cauchy_GP} in the case of \eqref{eq:GP} is hard to
implement. \color{black} Similarly, the existence result from
\cite{GalloGP} holds for more general nonlinearities than in
\eqref{eq:GP}, of the form $f(|u|^2)u$, but the function $f$ is
required to be at least $\mathcal C^2$ on $[0,\infty)$: in our case, $f(r)=\ln
r$ is not even continuous at the origin. 
\color{black}Even if a solution of the form $u\in 1+H^1$ is considered in
the one dimensional case (where the $H^1$-norm control the $L^\infty$
norm), it is not obvious to make a solution global in time, even if
$\|u_0-1\|_{H^1}\ll 1$. Indeed, the conservation of the energy in the
context of \eqref{logGP} reads
$d\mathcal{E}_{\textnormal{logGP}}/dt=0$, where
 \begin{equation*}
    \mathcal{E}_{\textnormal{logGP}} (u) \coloneqq \norm{\nabla
      u}_{L^2(\R^d)}^2 + \lambda \int_{\R^d} \( |u|^2\ln|u|^2-|u|^2+1\)\diff x.
\end{equation*}
Consider $F$ the antiderivative of the logarithm, such that $F(1) = 0$:
$F(y) = y \ln{y} - y + 1$ for $y > 0$, and $F(0) = 1$. The above
potential energy is the integral in space of $F(|u|^2)$. 
We note that $F \geq 0$ on $[0, \infty)$, and Taylor's formula yields
\begin{equation*}
  0\le F(y)=y^2\int_0^1\frac{1-s}{1+sy}\diff s\le y^2,
\end{equation*}
and so $0\le \mathcal{E}_{\textnormal{logGP}}(u)\le
\mathcal{E}_{\textnormal{GP}}(u)$. But this is not enough to solve
\eqref{logGP} in the energy space, neither locally or globally in
time. As a byproduct of our analysis, we will see that in low space
dimensions, $d\le 4$, the energy spaces for \eqref{logGP},
\begin{equation*}
    E_\textnormal{logGP} \coloneqq \{ v \in H^1_\textnormal{loc} (\mathbb{R}^d) \, | \, \mathcal{E}_\textnormal{logGP} (v) < \infty \},
\end{equation*} 
and
\eqref{eq:GP} coincide,
$E_{\textnormal{logGP}}= E_{\textnormal{GP}}$. Note that the energy
space $E_{\textnormal{GP}}$ was described very accurately in
\cite{Gerard_Cauchy_GP,PG08}. 
%

Since we will distinguish the notions of mild and weak solutions, we
clarify these two notions in the next definition. 
\begin{definition}
Let $u_0 \in  E_\textnormal{logGP}$, and $0\in I\subset \R$. We say that $u\in L^\infty_\textnormal{loc} (I,  E_\textnormal{logGP})$ is:\\
 $\bullet$ A weak solution to \eqref{logGP} on $I$ if it satisfies
 \begin{equation*}
       i  \partial_t u + \Delta u = \lambda u \ln{\abs{u}^2}
       \quad\text{in }\mathcal D'(I\times \R^d),
 \end{equation*}
and for any $\psi\in   \mathcal{C}_0^\infty(\R^d)$,
\begin{equation*}
  \lim_{t\to 0}\int_{\mathbb R^d}u(t,x)\psi (x)\diff x=\int_{\mathbb
    R^d}u_0(x)\psi (x)\diff x. 
\end{equation*}
$\bullet$  A mild solution to \eqref{logGP} on $I$ if it satisfies
  Duhamel's formula
\begin{equation*}
    u(t) = e^{i t \Delta} u_0 - i \lambda \int_0^t e^{i (t - s)
      \Delta} \( u(s) \ln{\abs{u (s)}^2} \) \diff s,\quad \forall t\in
    I.
\end{equation*}
\end{definition}

As we construct solutions to \eqref{logGP} as weak solutions, it is
sensible to ask whether or not these solutions are also mild
solutions, in the same fashion as in \cite{Segal63}. Our main result
regarding the Cauchy problem for \eqref{logGP} is the following:

\begin{theorem} \label{th:Cauchy_th_main}
    \color{black}Let $d\ge 1$.
For any $u_0 \in E_\textnormal{logGP}$, there exists a
    unique weak solution $u \in
    L^\infty_\textnormal{loc} (\mathbb{R}, 
    E_\textnormal{logGP})$ to
    \eqref{logGP}. Moreover, the 
    flow of \eqref{logGP} enjoys the following properties:
    \begin{itemize}
    \item The energy is conserved, $\mathcal{E}_{\textnormal{logGP}}
      (u(t))= \mathcal{E}_{\textnormal{logGP}}(u_0)$ for all $t\in
      \R$.
\item $u - u_0\in
    \mathcal{C}^0 (\mathbb{R}, L^2)$.
\item $u$ is also a mild  solution.
\item If $\Delta u_0 \in L^2$, then
    $u - u_0 \in W^{1, \infty}_\textnormal{loc}
    (\mathbb{R}, L^2) \cap L^\infty_\textnormal{loc} (\mathbb{R},
    H^2)$. 
    \end{itemize}
\end{theorem}
\color{black}
We emphasize that due to the singularity of the logarithm at the
origin, it is hopeless to look for a direct Picard fixed point
argument on Duhamel's formula in order to construct a solution. In
addition, like in the case of vanishing boundary conditions at
infinity evoked above, it is not clear that Strichartz estimates can
help in the construction of (local in time) solutions. 
\begin{rem}
  The parameter $\lambda>0$ can be set to  $1$, by
  considering the unknown function $v(t,x) = u\(\frac{t}{\lambda},\frac{x}{\sqrt
      \lambda}\)$. We choose to stick to the general notation $\lambda$ in
      order to keep track of the dependence on the nonlinear term more
      explicitly. 
\end{rem}
\color{black}
\subsection{Solitary and traveling waves}

There are many references regarding the existence, description and the
stability of solitary and traveling waves for \eqref{eq:GP}, as can
be seen for instance from the results evoked in \cite{volumeGP}. In
order to consider both solitary and traveling waves, we look for
solutions of the form
\begin{equation} \label{eq:trav_wave}
    u (t,x) = e^{i \omega t} \phi (x - ct),
\end{equation}
where $\omega \in \mathbb{R}$, $c \in \mathbb{R}^d$ and $\phi \in
E_\textnormal{logGP}$. In the case $c=0$, we classically call the
solution a solitary wave, and if $c\not =0$, we call it traveling
wave (stationary wave if $\omega=c=0$). From the mathematical point of
view, the study of traveling 
waves for \eqref{eq:GP} (with $\omega=0$) goes back to \cite{BeSa99},
and has known many 
developments since; see e.g. \cite{BeGrSa08,BeGrSa09,Ch12,ChMa17,Ma08,Ma13}
and references therein. \color{black}
We emphasize that in these generalizations, nonlinearities other than
cubic are considered, of the form $f(|u|^2)u$. However, it is always
assumed, among other assumptions, that $f$ is continuous on $[0,\infty)$: in the case considered
in the present paper, $f(r)=\ln r$ is unbounded at the origin, and
some work is needed. 
\color{black}

 In the one dimensional case, traveling waves for
\eqref{eq:GP} are either constant, or such that $0\le \abs{c} <\sqrt 2$, and
given explicitly by  
\begin{equation*}
  \phi_c(x) =
  \sqrt{1-\frac{c^2}{2}}\tanh\(\sqrt{1-\frac{c^2}{2}}\frac{x}{\sqrt
    2}\)+i\frac{c}{\sqrt 2},
\end{equation*}
up to space translation and a multiplicative constant of modulus $1$; see
e.g. \cite{BeGrSa08}. This solution $u(t,x)=\phi_c(x-ct)$ is usually
referred to as 
\emph{dark soliton}, and \emph{black soliton} (or \emph{kink solution})
in the case $c=0$ (the
only case where $\phi$ has a zero). In the case of \eqref{logGP}, we do
not have 
such an explicit formula, but somehow a very similar description. In
the general case $d\ge 1$, we have:

\begin{theorem} \label{th:value_omega}
Let $d\ge 1$.  \color{black}  If $\phi \in
E_\textnormal{logGP}$ and $u$, solution to \eqref{logGP}, is a
traveling wave of the form \eqref{eq:trav_wave}, then $\omega = 0$. 
\end{theorem}
This conclusion meets the one from \cite{Fa03}, a case where however
the assumptions on nonlinearities other than cubic are not written.
It is not  obvious that the logarithmic case is easily inferred,
as \cite{Fa98} (which is central in the argument of \cite{Fa03})
considers only polynomials nonlinearities. 
\color{black}
We now focus on the one-dimensional case. First, like for
\eqref{eq:GP}, if $|c|$ is too large, then there is no non-constant
traveling wave:
\begin{theorem} \label{th:trav_waves_1d}
    Let $d = 1$ and $c^2 \ge 2 \lambda$. Any solution to
    \eqref{logGP} of the form \eqref{eq:trav_wave} is  constant.
\end{theorem}
We now come to the description of non-constant traveling waves:

\begin{theorem} \label{th:sol_wave}
    Assume $d=1$ and $c$ such that $c^2 < 2 \lambda$. There exists a
    unique non-constant traveling wave in the following sense: there
    exists a traveling wave $\phi_c$ such that any non-constant
    traveling wave $u$ of the form \eqref{eq:trav_wave} is such that 
    \begin{equation*}
        \phi = e^{i \theta} \phi_c (\cdot- x_0),
    \end{equation*}
    for some constants $\theta, x_0 \in \mathbb{R}$. If $c\not =0$,
    $\phi_c$ never vanishes. 
    In the case $c=0$,  $\phi_0$ can be taken real-valued and
    increasing, and then 
    \begin{equation*}
        \lim_{x \to \pm \infty} \phi_0 (x) = \pm 1.
    \end{equation*}
 \end{theorem}
We emphasize that qualitatively, we obtain the same properties as the
black soliton, even though it is probably hopeless to get an explicit
expression in the case of \eqref{logGP}.  
\subsection{Content and notations}

In Section~\ref{sec:prelim}, we reduce the study of \eqref{logGP}
in the energy space, by analyzing the potential energy and
characterizing the energy space $E_\textnormal{logGP}$. In
Section~\ref{sec:regularity}, we  prove that any solution in the
energy space is unique, and that it is a mild solution. The heart of Theorem~\ref{th:Cauchy_th_main} is established in
Section~\ref{sec:construction}, where a weak solution is constructed
by an approximation procedure. In Section~\ref{sec:higher}, we show
that $H^2$ regularity is propagated by the flow, completing the proof
of Theorem~\ref{th:Cauchy_th_main}. We prove general results regarding
 solitary and
traveling waves in Section~\ref{sec:dynamics}, in particular
Theorem~\ref{th:value_omega}. In Section~\ref{sec:solitary}, we
analyze stationary waves in the one-dimensional case
(Theorem~\ref{th:sol_wave} in the case $c=0$). One-dimensional traveling waves are
studied in Section~\ref{sec:traveling}, where we complete the proof of
Theorems~\ref{th:sol_wave} and \ref{th:trav_waves_1d}. 
We
conclude this paper by open questions in Section~\ref{sec:open}.
\smallbreak

We define the $L^2$-bracket by
\begin{equation*}
    \langle f, g \rangle = \int_{\R^d} f \overline{g} \diff x.
\end{equation*}
For $k$ an integer, we denote by  $\mathcal{C}^k$ the class of $k$ times
continuously differentiable functions, and if $0<s<1$, by
$\mathcal{C}^{k,s}$ the class of $k$ times
continuously differentiable functions whose derivative of order $k$ is
H\"older continuous with exponent $s$. If $s=1$, $\mathcal{C}^{k,1}$ is
the class of $k$ times 
continuously differentiable functions whose derivative of order $k$ is
Lipschitzian.


\section{Preliminary results}
\label{sec:prelim}

\subsection{Generalities}

First, we recall some properties about the logarithm.
The first property was discovered in \cite{cazenave-haraux}, and is
crucial for uniqueness issues related to \eqref{logGP}. 
\begin{lem}[{\cite[Lemma~1.1.1]{cazenave-haraux}}] \label{lem:log_inequality}
There holds

\begin{equation*}
    \abs{\Im \left( (z_2 \ln \abs{z_2}^2 - z_1 \ln \abs{z_1}^2) (\overline{z_2} - \overline{z_1}) \right)} \leq 2 \abs{z_2 - z_1}^2, \qquad \forall z_1, z_2 \in \mathbb{C}.
\end{equation*}
\end{lem}

The next lemma measures the continuity of the nonlinearity in \eqref{logGP}:

\begin{lem} \label{lem:prop_lip_log}
    For all $\varepsilon \in (0, 1)$ and $x, y \in \mathbb{C}$,
    \begin{equation*}
        \abs{x \ln{\abs{x}^2} - y \ln{\abs{y}^2}} \leq C \Bigl( \abs{x}^\varepsilon \abs{\ln{\abs{x}}} + \abs{y}^\varepsilon \abs{\ln{\abs{y}}} \Bigr) \abs{x - y}^{1 - \varepsilon} + 2 \abs{x - y}.
    \end{equation*}
\end{lem}

\begin{proof}
    Let $x, y \in \mathbb{C}$ and assume without loss of generality that $\abs{x} \leq \abs{y}$. We also assume $x \neq 0$, otherwise the estimate is obvious.
    Then, we have
    \begin{equation*}
        \abs{x \ln{\abs{x}^2} - x \ln{\abs{y}^2}} = 2 \abs{x} \abs{\ln{\abs{x}} - \ln{\abs{y}}} 
            \leq 2 \abs{x} \frac{\abs{\abs{x} - \abs{y}}}{\abs{x}} \leq 2 \abs{x - y}.
    \end{equation*}
    On the other hand,
    \begin{equation*}
        \abs{x \ln{\abs{y}^2} - y \ln{\abs{y}^2}} \leq 2 \abs{x - y} \abs{\ln{\abs{y}}} \leq 2^{1+\varepsilon} \abs{x - y}^{1 - \varepsilon} \abs{y}^\varepsilon \abs{\ln{\abs{y}}},
    \end{equation*}
    hence the conclusion.
\end{proof}

Then, we also recall a result about norms on finite dimensional subspace of $L^2$.

\begin{lem} \label{lem:equiv_norm_finite_dim}
    If $\{ f_k \}_{0\leq k \leq n}$ is a finite family of $L^2$ linearly
    independent functions, then there exists $C > 0$ such that, for all $\lambda_k \in \mathbb{C}$,
    \begin{equation*}
        \max_{0\leq k \leq n} \abs{\lambda_k} \leq C \norm{\sum_{k = 0}^n \lambda_k f_k}_{L^2}.
    \end{equation*}
\end{lem}

\begin{proof}
  As the functions $\{ f_k \}_{0\leq k \leq n}$ are linearly independent, the map
 $(\lambda_k)_{0\leq k \leq n} \mapsto \left\|\sum_{k = 0}^n \lambda_k f_k\right\|_{L^2}$ defines
 a norm on $\C^{n+1}$. The conclusion readily follows, since the left
 hand side of the conclusion is the $\ell^\infty$ norm on
 $\C^{n+1}$, and all norms are
 equivalent on finite dimensional spaces. 
\end{proof}

\subsection{Potential energy}

We start by a result on the potential energy $\mathcal{E}_{\textnormal{\textnormal{pot}}}$, given by
\begin{equation*}
    \mathcal{E}_{\textnormal{\textnormal{pot}}} (v) \coloneqq \int_{\R^d} \Bigl( \abs{v}^2 \ln{\abs{v}^2} - \abs{v}^2 + 1 \Bigr) \diff x.
\end{equation*}
Let
\begin{equation*}
    E_\textnormal{\textnormal{pot}} (v) \coloneqq \int_{\R^d}  ( \abs{v} - 1 )^2 \ln{(2 + \abs{v})} \diff x.
\end{equation*}
We know that both previous quantities are non-negative. The following result shows that they are actually equivalent.

\begin{lem} \label{lem:energy_equiv}
    There exists $K_0 > 0$ such that for all $v \in E_\textnormal{logGP}$,
    \begin{equation*}
        \frac{1}{K_0} E_\textnormal{pot} (v) \leq \mathcal{E}_\textnormal{pot} (v) \leq K_0 E_\textnormal{pot} (v).
    \end{equation*}
\end{lem}

\begin{proof}
Taylor formula yields, for $y\ge 0$,
\begin{equation*}
  y^2\ln y^2-y^2+1 = 4 \(y-1\)^2\int_0^1
  \(1+\ln\(1-s+sy\)\)(1-s)ds.
\end{equation*}
Distinguishing $y<1$ and $y>1$, we get:
\begin{equation*}
 \ln \(2+y\)\lesssim \int_0^1\(1+\ln\(1-s+sy\)\)(1-s)ds\lesssim
 \ln \(2+y\),
\end{equation*}
and the result follows, by considering $y=|v|$.
 \end{proof}

This second functional $E_\textnormal{pot}$ appears to be more
convenient than $\mathcal{E}_\textnormal{pot}$. For instance, since
$\ln{2} \leq \ln{(2 + \abs{x})} \leq \ln{2} + C_\delta \abs{\abs{x} -
  1}^\delta$ for all $\delta > 0$, we can relate $E_\textnormal{pot}$
to the $L^2$-norm of $\abs{v} - 1$, but also to any $L^p$-norm, for
$p > 2$ as follows:

\begin{lem} \label{lem:E_pot_Lp}
    Let $p>2$. There exists $C_p$ such that for all $v \in E_\textnormal{logGP}$, 
    \begin{equation*}
        \ln{(2)} \, \norm{\abs{v} - 1}_{L^2}^2 \leq E_\textnormal{pot} (v) \leq \ln{(2)} \, \norm{\abs{v} - 1}_{L^2}^2 + C_p \norm{\abs{v} - 1}_{L^p}^p.
    \end{equation*}
      \end{lem}


    %

\subsection{Energy space}

As for the energy space, Lemma~\ref{lem:energy_equiv}  leads to
\begin{equation*}
    E_\textnormal{logGP} = \{ v \in H^1_\textnormal{loc} (\mathbb{R}^d) \, | \, \nabla v \in L^2(\R^d) \, \textnormal{and} \, E_\textnormal{pot} (v) < \infty \}.
\end{equation*}
We prove an even more explicit description of $E_\textnormal{logGP}$:

\begin{lem} \label{lem:E_logGP_descr}
The energy space is characterized by:
\begin{align*}
  E_\textnormal{logGP} &= \{ v \in H^1_\textnormal{loc} (\R^d)\, | \, \nabla
  v \in L^2(\R^d) \, \textnormal{and} \, \abs{v} - 1 \in L^2 (\R^d)\} \\
&= \{ v \in
  H^1_\textnormal{loc} (\R^d)\, | \, \nabla v \in L^2(\R^d) \, \textnormal{and} \,
  \abs{v} - 1 \in H^1(\R^d) \} .
\end{align*}
 Moreover, there exists $C>0$ such that for all $v \in E_\textnormal{logGP}$, $\norm{\abs{v} - 1}_{H^1}^2 \leq C \mathcal{E}_\textnormal{logGP} (v)$.
\end{lem}

\begin{proof}
 First, we show that $E_\textnormal{logGP} \subset \{ v \in
 H^1_\textnormal{loc} \, | \, \nabla v \in L^2 \, \textnormal{and}
 \, \abs{v} - 1 \in L^2 \}$. 

  Let $v \in E_\textnormal{logGP}$: $\norm{\nabla v}_{L^2} <
 \infty$, and $\mathcal{E}_\textnormal{pot} (v) < \infty$, which
 is equivalent to $E_\textnormal{pot} (v) <
 \infty$. Lemma~\ref{lem:E_pot_Lp} yields $\abs{v} - 1 \in L^2$. 
 \medskip
    
Conversely, let $v \in {H}^1_\textnormal{loc}$ such that $\abs{v} -  1
,\nabla v\in L^2$. Then $\nabla (\abs{v} - 1) = \Re\( \frac{\bar
  v}{\abs{v}} \nabla v\)$ and $\abs{\nabla (\abs{v} - 1)} \leq \abs{\nabla
  v}$, which shows that $f = \abs{v} - 1 \in H^1$.  Since $f \geq -1$,
we readily have 
    \begin{equation*}
        E_\textnormal{pot} (v) = \int_{\R^d} \abs{f}^2 \ln{(3 + f)} \diff x
        \leq \int_{\R^d} \abs{f}^2 \ln{(3 + \abs{f})} \diff x \lesssim \int_{\R^d} \abs{f}^2 \(1+ \abs{f}^\eps\) \diff x < \infty,
    \end{equation*}
where $\eps>0$ is arbitrarily small, and where we have used Sobolev
embedding to conclude.
    \medskip
    
 From the previous arguments, if $v \in E_\textnormal{logGP}$, 
    the $L^2$-norm of $|v|-1$ can be estimated thanks to Lemma~\ref{lem:E_pot_Lp} and Lemma~\ref{lem:energy_equiv}, whereas $\abs{\nabla (\abs{v} - 1)} \leq \abs{\nabla v}$, hence  the conclusion.
\end{proof}

\begin{cor} \label{cor:v_ln_v_Lp}
  For all $\delta\ge 1$, $\eps>0$, and $p$ such that 
  \begin{equation} \label{eq:p_cases}
        p,p+\eps \in 
        \begin{cases}
            [2, \infty] \quad \text{if } d = 1, \\
            [2, \infty) \quad \text{if } d = 2, \\
            [2, 2^*] \quad \text{if } d \geq 3,\quad 2^* \coloneqq \frac{2d}{d-2},
        \end{cases}
    \end{equation}
 there exists $C_{p,\eps,\delta}$ such that 
  if $v \in E_\textnormal{logGP}$, then $v \abs{\ln{\abs{v}^2}}^\delta
  \in L^p(\R^d)$ and 
    \begin{equation*}
        \norm{v \abs{\ln{\abs{v}^2}}^\delta}_{L^p} \leq C_{p, \varepsilon, \delta} \, (\mathcal{E}_\textnormal{logGP} (v)^{1/2} + \mathcal{E}_\textnormal{logGP} (v)^{(1 + \eps)/2}).
    \end{equation*}
  \end{cor}

\begin{proof}
    We know that $\abs{v} - 1 \in H^1$ and, since $\delta \geq 1$,
    \begin{equation*}
        \abs{ v} \abs{\ln{\abs{v}}^2}^\delta  \le C_\delta \abs{\abs{v} -
          1} (\ln{(2 + \abs{\abs{v} - 1})})^\delta \leq
        C_{\varepsilon, \delta} (\abs{\abs{v} - 1} + \abs{\abs{v} -
          1}^{1+\frac{\varepsilon}{p}}). 
    \end{equation*}
This inequality is easily established, for instance by considering
separately the
regions 
$\{|v|<1/2\}$, $\{|v|>2\}$, and $\{1/2\le |v|\le 2\}$. 
    The result then follows from Sobolev embedding,
    Lemma~\ref{lem:E_logGP_descr} and Lemma~\ref{lem:energy_equiv}.
\end{proof}
\begin{rem}
  This corollary reveals an important difference with the case of
  vanishing boundary condition at infinity: in
  \cite{GuLoNi10,carlesgallagher}, momenta in $L^2$
  ($\||x|^\alpha u\|_{L^2}$ for some $\alpha\in (0,1]$) are considered
  in order to control the nonlinearity in the region $\{|u|<1\}$. The
  situation is obviously different in the case of nontrivial boundary
  condition at infinity, and the above corollary is crucial for the
  rest of this paper. 
\end{rem}

Lemma~\ref{lem:E_logGP_descr}  shows
that for $d\le 4$, $E_\textnormal{logGP} = E_\textnormal{GP}$, in view of the
identity
\begin{equation*}
  |u|^2-1 = \(|u|-1\)\(|u|-1+2\),
\end{equation*}
and the Sobolev embedding $H^1(\R^d)\subset L^4(\R^d)$ for $d\le 4$. 
In particular,
G\'erard \cite{Gerard_Cauchy_GP} proved that $E_\textnormal{GP} + H^1 \subset
E_\textnormal{GP}$ for $d \ge 1$, and therefore the same holds for
$E_\textnormal{logGP}$ when $d\le 4$. This is also true in all dimensions : 

\begin{lem} \label{lem:Energy_plus_H1}
    Let $d \geq 1$. For all $v \in E_\textnormal{logGP}$ and for all
    $f \in H^1$,  we have $v + f \in E_\textnormal{logGP}$. In
    addition, if $p > 2$ satisfies \eqref{eq:p_cases}, there exists
    $C_p$ such that for every such $v$ and $f$,
    \begin{equation*}
        \mathcal{E}_\textnormal{logGP} (v + f) \leq C_p
        \mathcal{E}_\textnormal{logGP} (v) + C_p
        \mathcal{E}_\textnormal{logGP} (v)^p + C _p\norm{f}_{H^1}^2 +
        C_p \norm{f}_{H^1}^p. 
    \end{equation*}
\end{lem}

\begin{proof}
    Let $v \in E_\textnormal{logGP}$ and $f \in H^1$. We know that
    $\nabla v ,\nabla f\in L^2$,  so $\nabla(v + f) \in L^2$. From Lemma \ref{lem:E_logGP_descr}, there only remains to show that $\abs{v + f} - 1 \in L^2$. For this, we have for all $x, y \in \mathbb{C}$
    \begin{equation*}
        \abs{x + y} = \abs{x} + \int_0^1 \frac{x + r y}{\abs{x + r y}} \cdot y \diff r.
    \end{equation*}
    Applying this equality to $u$ and $f$, we get
    \begin{equation*}
        \abs{u + f} - 1 = \abs{u} - 1 + \int_0^1 \frac{u + r f}{\abs{u + r f}} \cdot f \diff r,
    \end{equation*}
    hence
    \begin{equation*}
        \abs{\abs{u + f} - 1} \leq \abs{\abs{u} - 1} + \abs{f},
    \end{equation*}
    and, for any $p \in [1, \infty]$,
    \begin{equation*}
        \norm{\abs{u + f} - 1}_{L^p} \leq \norm{\abs{u} - 1}_{L^p} + \norm{f}_{L^p}.
    \end{equation*}
    Moreover, Lemmas \ref{lem:energy_equiv} and \ref{lem:E_pot_Lp} yield, for any $p > 2$ satisfying \eqref{eq:p_cases},
    \begin{align*}
        \mathcal{E}_\textnormal{logGP} (v + f) &\leq \norm{\nabla v + \nabla f}_{L^2}^2 + C \norm{\abs{v + f} - 1}_{L^2}^2 + C_p \norm{\abs{v + f} - 1}_{L^p}^p \\
            &\lesssim \norm{\nabla v}_{L^2}^2 +  \norm{\nabla f}_{L^2}^2 + \norm{\abs{v} - 1}_{L^2}^2 + \norm{f}_{L^2}^2 + \norm{\abs{v} - 1}_{L^p}^p +  \norm{f}_{L^p}^p \\
            &\lesssim \mathcal{E}_\textnormal{logGP} (v) +  \norm{\abs{v} - 1}_{H^1}^p +  \norm{f}_{H^1}^2 + \norm{f}_{H^1}^p \\
            &\lesssim \mathcal{E}_\textnormal{logGP} (v) +  \mathcal{E}_\textnormal{logGP} (v)^p + \norm{f}_{H^1}^2 +  \norm{f}_{H^1}^p. \qedhere
    \end{align*}
\end{proof}

We conclude by adapting \cite[Lemma~1]{Gerard_Cauchy_GP}. 
Introduce the Zhidkov space
\begin{equation*}
 X^1(\R^d) = \{u\in L^\infty(\R^d),\ \nabla u\in L^2(\R^d)\}.  
\end{equation*}
\begin{lem}\label{lem:zhidkov}
  Let $d\ge 1$. We have $E_\textnormal{logGP}\subset X^1(\R^d)+H^1(\R^d)$.
\end{lem}
\begin{proof}
  Proceeding like in the proof of \cite[Lemma~1]{Gerard_Cauchy_GP},
  consider $\chi\in \mathcal C_0^\infty(\C)$ a cutoff function such
  that $0\le \chi\le 1$, $\chi(z)=1$ for $|z|\le 2$ and $\chi(z)=0$
  for $|z|\ge 3$. Decompose $u\in E_\textnormal{logGP}$ as 
  \begin{equation*}
    u=u_1+u_2,\quad u_1=\chi(u)u,\quad u_2=\(1-\chi(u)\)u.
  \end{equation*}
Then $\|u_1\|_{L^\infty}\le 3$, and since $|u|\ge 2$ on the support of
$u_2$, for any $\delta>1$, we may find $C_\delta$ such that 
\begin{equation*}
  |u_2|\le C_\delta \left||u|-1\right|^\delta.
\end{equation*}
Since $|u|-1\in H^1(\R^d)$ from Lemma~\ref{lem:E_logGP_descr},
$|u|-1\in L^{2\delta}(\R^d)$ by Sobolev embedding (provided that
$\delta\le\frac{d}{d-2}$ if $d\ge 3$), hence $u_2\in L^2$. The
properties $\nabla u_1,\nabla u_2\in L^2$ are straightforward, and we
refer to \cite{Gerard_Cauchy_GP} for details.
\end{proof}


\section{Regularity and uniqueness}
\label{sec:regularity}

We begin by proving the following point in Theorem \ref{th:Cauchy_th_main}:

\begin{lem} \label{lem:u_in_plusH1}
    Let $I\ni 0$ be a time interval. If $u \in L^\infty_\textnormal{loc} (I,
    E_\textnormal{logGP})$ (that is, $t\mapsto \mathcal{E}_\textnormal{logGP}
    (u(t)) \in L^\infty_\textnormal{loc} (I)$)  is a
      weak solution to \eqref{logGP} on $I$, then $u - u_0 \in
    \mathcal C^0 (I, L^2)$.
\end{lem}

\begin{proof}
    Since $u \in L^\infty_\textnormal{loc} (I, E_\textnormal{logGP})$,
    we know from Corollary~\ref{cor:v_ln_v_Lp} that $u \ln{\abs{u}^2}
    \in L^\infty_\textnormal{loc} (I, L^2)$. Moreover, $\Delta u \in
    L^\infty_\textnormal{loc} (I, H^{-1})$. Therefore, \eqref{logGP}
    yields $\partial_t u =
    \partial_t (u - u_0) \in L^\infty_\textnormal{loc} (I, H^{-1})$,
    with $u(0) - u_0 = 0$. This proves that $u - u_0 \in \mathcal{C}^0
    (I, H^{-1})$. Since $\nabla(u - u_0) \in L^\infty_\textnormal{loc}
    (I, L^2)$, we obtain the result by interpolation. 
\end{proof}

This result allows us to infer the uniqueness of the solution of \eqref{logGP}.

\begin{theorem} \label{th:regu_sol}
    Let $u_0 \in E_\textnormal{logGP}$ and $I\ni 0$ be an
    interval. There exists at most one  weak solution
    $u\in L^\infty_{\rm loc}(I,  E_\textnormal{logGP})$
    to \eqref{logGP}. 
\end{theorem}

\begin{proof}
    Let $u,v \in L^\infty (I, E_\textnormal{logGP})$ solve
    \eqref{logGP}. Then $u (t) - u_0$ and $v (t) - u_0$ are continuous
    from $I$ to $L^2$  from Lemma \ref{lem:u_in_plusH1}. Therefore, $w
    \coloneqq u - v \in \mathcal C^0(I;L^2)$. This error satisfies (in $\mathcal{D}' (I \times \mathbb{R}^d)$)
    \begin{equation*}
        i \partial_t w + \Delta w = \lambda (u \ln{\abs{u}^2} - v \ln{\abs{v}^2}).
    \end{equation*}
    Since on the one hand $\Delta w \in L^\infty (I, H^{-1})$, and on
    the other hand $u \ln{\abs{u}^2},v \ln{\abs{v}^2}\in L^\infty (I,
    L^2)$, the previous equality is also satisfied in $L^\infty (I,
    H^{-1})$. As $w \in L^\infty (I, H^1)$, we can take the $H^{-1}
    \times H^1$ bracket against $w$ of the previous equality, which
    yields 
    \begin{equation} \label{eq:pde_w}
        i \langle \partial_t w, w \rangle_{H^{-1}, H^1} - \norm{\nabla w}_{L^2}^2 = \lambda \int_{\R^d} (u \ln{\abs{u}^2} - v \ln{\abs{v}^2}) (\overline{u - v}) \diff x.
    \end{equation}
    Thanks to Lemma \ref{lem:log_inequality}, we have
    \begin{equation*}
        \abs{\Im \int_{\R^d} (u \ln{\abs{u}^2} - v \ln{\abs{v}^2}) (\overline{u - v}) \diff x} \leq 2 \norm{u - v}_{L^2}^2 = 2 \norm{w}_{L^2}^2.
    \end{equation*}
    Therefore, taking the imaginary part of \eqref{eq:pde_w} leads to
    \begin{equation*}
        \abs{\Re \langle \partial_t w, w \rangle_{H^{-1}, H^1}} = \frac{1}{2} \abs{\frac{\diff}{\diff t} \norm{w}_{L^2}^2} \leq 2 \lambda \norm{w}_{L^2}^2.
    \end{equation*}
    Since $w(0) = 0$, Gronwall lemma concludes the proof.
\end{proof}

We end this section with a link between regularity and mild solution.

\begin{lem} \label{lem:mild_sol}
    Let $I\ni 0$ be an open interval of $\mathbb{R}$, $u_0 \in
    E_\textnormal{logGP}$ and $u \in L^\infty_\textnormal{loc} (I,
    E_\textnormal{logGP})$ be a weak solution to
      \eqref{logGP} on $I$. Then $u$ is a mild solution, and 
    \begin{equation*}
      \int_0^t e^{- i s \Delta} \Bigl( u(s) \ln{\abs{u (s)}^2} \Bigr)
      \diff s \in H^1(\R^d)\quad \text{for all }t\in I.
    \end{equation*}
\end{lem}

\begin{proof}
    From Corollary~\ref{cor:v_ln_v_Lp}, we have $u \ln{\abs{u}^2} \in L^\infty_\textnormal{loc} (I, L^2)$, thus we can define
    \begin{equation*}
        v_\textnormal{NL} (t) = - i \lambda \int_0^t e^{-i s \Delta} u (s) \ln{\abs{u (s)}^2} \diff s \in \mathcal{C}^0 (I, L^2) \cap W^{1, \infty}_\textnormal{loc} (I, H^{-2}).
    \end{equation*}
    On the other hand, since $\nabla u_0 \in L^2$ and $u_0 \in
    L^\infty + L^2$ (from Lemma~\ref{lem:zhidkov}), we can define
    $e^{i t \Delta} u_0$ and we have $e^{i t \Delta} u_0 - u_0 \in
    \mathcal{C}^0 (I, L^2) \cap \mathcal{C}^1 (I, H^{-1})$ (see
    \cite{Gerard_Cauchy_GP}). 
    Therefore,
    \begin{equation*}
        v \coloneqq e^{i t \Delta} u_0 +e^{i t \Delta}
        v_\textnormal{NL} (t) \in u_0 + \mathcal{C}^0 (I, L^2) \cap
        W^{1, \infty}_\textnormal{loc}  (I, H^{-2}),
    \end{equation*}
    and we can compute
    \begin{align*}
      \partial_t v &= i \Delta e^{i t \Delta} \Bigl[ u_0 - i \lambda
                     \int_0^t e^{- i s \Delta} u(s) \ln{\abs{u (s)}^2}
                     \diff s \Bigr] \\
&\quad+ e^{i t \Delta} \partial_t \Bigl[ u_0 - i \lambda \int_0^t e^{- i s \Delta} u(s) \ln{\abs{u (s)}^2} \diff s \Bigr] \\
            &= i \Delta v + e^{i t \Delta} \Bigl[ - i \lambda e^{- i t \Delta} u(t) \ln{\abs{u (t)}^2} \Bigr] \\
            &= i \Delta v - i \lambda u(t) \ln{\abs{u (t)}^2},
    \end{align*}
    where the equality  is to be understood in
    $L^\infty_\textnormal{loc} (I, H^{-2})$. Let $w = u - v$. Then, $w
    \in \mathcal{C}^0 (I, L^2)$ from Lemma \ref{lem:u_in_plusH1} along
    with the previous arguments, and with the previous equality, we
    have,  in $L^\infty_\textnormal{loc} (I, H^{-2})$,
    \begin{equation*}
        \partial_t w - i \Delta w = 0.
    \end{equation*}
    Since $w (0) = 0$ by construction, we get $w = 0$, which proves the mild formulation.
    Last,  we know that  $\nabla (u(t) - u_0) \in L^2$ and $\nabla (e^{i t \Delta} u_0 - u_0) \in L^2$ for all $t \in I$, therefore
    \begin{equation*}
        \nabla v_\textnormal{NL} (t) \in L^2, \qquad \forall t \in I.
    \end{equation*}
    Since we already know that $v_\textnormal{NL} (t) \in L^2$, we get the conclusion.
\end{proof}


\section{Construction of a solution}
\label{sec:construction}

In this section, we construct a global solution $u \in L^\infty
(\mathbb{R}, \mathcal{E}_\textnormal{logNLS})$ to \eqref{logGP}. We
adapt the method used by Ginibre and Velo
\cite{Ginibre_Velo__Cauchy_NLS_revisited} to construct global weak
solutions to non-linear Schrödinger equations (NLS) by
compactness. Here, the framework is different since the solution is
not in $L^2$. However, we know that we should have $u (t) - u_0 \in
H^1$. Therefore, we will approximate the solution on $u_0 + X_m$ with
$X_m$ a sequence of finite dimensional linear subspaces approximating $H^1$.
\begin{rem}
  In \cite{carlesgallagher}, for vanishing boundary condition at
  infinity, another approximation is considered,
  consisting in removing the singularity of the logarithm by
  saturating the nonlinearity, for $\eps>0$,
  \begin{equation}\label{eq:logNLSreg}
       i  \partial_t u^\eps + \Delta u^\eps = \lambda u^\eps \ln\({\abs{u^\eps}^2+\eps}\), \quad
   u^\eps_{\mid t=0} =u_0,
  \end{equation}
and letting  $\eps\to 0$. This approximation has the advantage of working
whichever the sign of $\lambda$ is, while the approach introduced
initially in \cite{cazenave-haraux} (see also \cite{HayashiM2018})
seems to be bound to the 
nondispersive case $\lambda<0$. 
\color{black}
In the case of Gross-Pitaevskii equation,  integrability  constraints
are different, and the corresponding
potential energy is given by
\begin{equation*}
  \lambda\int_{\R^d} \( (|u^\eps|^2+\eps)\ln\(|u^\eps|^2+\eps\) -
  \(1+\ln(1+\eps)\)\(|u^\eps|^2-1\) - (1+\eps)\ln(1+\eps)\).
\end{equation*}
We believe that this strategy should work,  possibly with long
computations; we rather choose an approach
involving an energy independent of the approximation. 
\color{black}
\end{rem}

\subsection{Finite dimensional approximation}

Let $\{ w_j \}_{j \in \mathbb{N}}$ a Hilbert basis of $L^2$ with
all $w_j \in H^1$ and $m \in \mathbb{N}$. One may think for instance
of Hermite functions, all the more since in
Lemma~\ref{lem:dt_um_unif_bound} below, in order to prove the
propagation of $\dot H^2$ regularity, we further require $w_j\in
H^2$. We take $X_m \coloneqq
\operatorname{Vect} (w_j)_{j \leq m}$ and look for an approximation
$u_m$ of the form 
\begin{equation} \label{eq:form_u_m}
    u_m (t,x) = u_0(x) + \sum_{k=0}^m g_{m, k} (t) w_k(x) \eqqcolon u_0(x) +
    \varphi_m (t,x), 
\end{equation}
satisfying 
\begin{equation} \label{eq:eq_u_m}
    \< w_j, i \partial_t u_m + \Delta u_m - \lambda u_m
    \ln{\abs{u_m}^2} \>_{H^1, H^{-1}} = 0,\quad 0\le j\le m,\ \forall
    t\in \R,
\end{equation}
\color{black}where $\<\varphi,\psi\>_{H^1, H^{-1}} =\int_{\R^d} \varphi\bar \psi$ for
  Schwartz functions, \color{black} 
with initial condition $u_m (t) = u_0$, which is equivalent to $g_{m, k} (0) = 0$ for all $0 \leq k \leq m$.
By substitution, \eqref{eq:eq_u_m} is equivalent to
\begin{equation} \label{eq:ode_g_mk}
  \begin{aligned}
        i \dot{\bar g}_{m, j} - \< \nabla w_j, \nabla u_0 \> - &\sum_{k = 0}^m
   \bar  g_{m, k} (t) \< \nabla w_j, \nabla w_k \> \\
&= \lambda \< w_j, f \(
    u_0 + \sum_{k=0}^m g_{m, k} (t) w_k \) \>, 
  \end{aligned}
\end{equation}
where $f(x) = x \ln{\abs{x}^2}$.
First, we study the last term.

\begin{lem} \label{lem:cont_log_term_um}
 The map   $\varphi \mapsto \< w_j, f (\varphi) \>$ is well defined
 and continuous from $u_0 + H^1$ into $\mathbb{C}$ for any $j \in
 \mathbb{N}$. More precisely, it is $\mathcal{C}^{0, \varepsilon} (u_0
 + H^1, \mathbb{C})$ for all $\varepsilon \in (0, 1)$. 
\end{lem}

\begin{proof}
    By Corollary~\ref{cor:v_ln_v_Lp} along with
    Lemma~\ref{lem:Energy_plus_H1}, we know that the right-hand side
    is well 
    defined. Moreover, at fixed $j$, taking $\varphi, \psi \in u_0 +
    H^1$, we get 
    \begin{align*}
        \abs{\< w_j, f ( \varphi ) \> - \< w_j, f ( \psi ) \>} &= \abs{\< w_j, f ( \varphi ) - f ( \psi ) \>} \leq \int_{\R^d} \abs{w_j} \abs{f ( \varphi ) - f ( \psi )}
              \diff x \\ 
            &\leq C \int_{\R^d} \abs{w_j} \Bigl( \abs{\varphi}^\varepsilon
              \abs{\ln{\abs{\varphi}}} + \abs{\psi}^\varepsilon
              \abs{\ln{\abs{\psi}}} \Bigr) \abs{\varphi - \psi}^{1 -
              \varepsilon} \diff x \\
&\quad+ 2 \int_{\R^d} \abs{w_j} \abs{\varphi -
              \psi} \diff x, 
    \end{align*}
    for any $\varepsilon \in (0, 1)$, where we have used Lemma \ref{lem:prop_lip_log}.
    The second term of the last inequality is obviously Lipschitzian with
    respect to the $L^2$-norm of the difference. As for the first
    term, by H\"older inequality with exponents $2$,
    $\frac{2}{\varepsilon}$ and $\frac{2}{1 - \varepsilon}$, we
    control it by  
    \begin{equation*}
 \norm{w_j}_{L^2} \norm{\varphi - \psi}_{L^2}^{1 - \varepsilon} \Bigl( \norm{\varphi \abs{\ln{\abs{\varphi}}}^\frac{1}{\varepsilon}}_{L^2}^\varepsilon + \norm{\psi \abs{\ln{\abs{\psi}}}^\frac{1}{\varepsilon}}_{L^2}^\varepsilon \Bigr).
    \end{equation*}
Now,    $\norm{\varphi \abs{\ln{\abs{\varphi}}}^\frac{1}{\varepsilon}}_{L^2}$ and $\norm{\psi \abs{\ln{\abs{\psi}}}^\frac{1}{\varepsilon}}_{L^2}$ are locally bounded thanks to Corollary \ref{cor:v_ln_v_Lp} and Lemma \ref{lem:Energy_plus_H1} once again. These estimates yield the conclusion.
\end{proof}

\begin{lem} \label{lem:exist_sol_app}
    There exists a solution $\{ g_{m, k} \}_{k \leq m}$ to
    \eqref{eq:ode_g_mk} on a maximal time interval $(- T_m, T^m)$,
    meaning that $T^m < \infty$ if and only if
    \begin{equation*}
      \limsup_{t \rightarrow
      T^m} \sup_{j\le m}\abs{g_{m, j} (t)} = \infty.
    \end{equation*}
     (And    similarly for $T_m$.) Moreover, the $g_{m, k}$'s are
    $\mathcal{C}^{1, 1 - \varepsilon} (- T_m, T^m)$ for all
    $\varepsilon \in (0, 1)$. 
\end{lem}

\begin{proof}
    The last term of \eqref{eq:ode_g_mk} is continuous (and even
    $\mathcal{C}^{0, \varepsilon}$ for all $\varepsilon \in (0, 1)$)
    with respect to the $g_{m, k}$'s thanks to Lemma~\ref{lem:cont_log_term_um}, and so are obviously the other
    terms. The conclusion comes from Peano Theorem and the fact that
    \eqref{eq:ode_g_mk} is an autonomous ODE.  Note that
    we invoke Peano Theorem, and not Cauchy-Lipschitz Theorem, since
    the nonlinearity $f$ is continuous, but not locally Lipschitzian,
    due to the singularity of the logarithm at the origin.
\end{proof}

We will then show that such a solution is global. First, we prove an intermediate result.

\begin{lem} \label{lem:nabla_lin_indep}
    For all $m \in \mathbb{N}$, $(\nabla w_k)_{k \leq m}$ are linearly independent.
\end{lem}

\begin{proof}
    Let $\lambda_k \in \mathbb{C}$ such that $\sum_{k=0}^m \lambda_k
    \nabla w_k = 0$. Then $\nabla \psi = 0$, with $\psi \coloneqq
    \sum_{k=0}^m \lambda_k w_k \in H^1$. Therefore, $\psi = 0$, and we
    conclude by the fact that $\{ w_k \}$ is a  Hilbert basis of $L^2$.
\end{proof}

\begin{lem} \label{lem:appr_sol_global}
    The solution given by Lemma \ref{lem:exist_sol_app} is
    global. Moreover, it satisfies 
    \begin{equation*}
      \mathcal{E}_\textnormal{logGP} (u_m (t)) =
      \mathcal{E}_\textnormal{logGP} (u_0), \quad\text{and}
      \quad\norm{\nabla \varphi_m (t)}_{L^2} \leq 2
      \sqrt{\mathcal{E}_\textnormal{logGP} (u_0)}\quad \text{for all
      }t \in \R.
    \end{equation*}
\end{lem}

\begin{proof}
    For this, we only have to prove that all the $g_{m, k}$s are (locally) bounded.
    Come back to \eqref{eq:eq_u_m}: by multiplying by $\dot{g}_{m, j}$
    and summing over $j \leq m$, we get, in view of \eqref{eq:form_u_m},
    \begin{equation*}
        \< \partial_t u_m, i \partial_t u_m +\Delta u_m - \lambda u_m \ln{\abs{u_m}^2} \> = 0.
    \end{equation*}
    By taking the real part of this equation, we obtain
    \begin{equation*}
        \frac{\diff}{\diff t} \mathcal{E}_\textnormal{logGP} (u_m (t)) = 0.
    \end{equation*}
    Therefore, $\mathcal{E}_\textnormal{logGP} (u_m (t)) =
    \mathcal{E}_\textnormal{logGP} (u_0)$ for all $t \in (- T_m,
    T^m)$. In particular, $\norm{\nabla u_m (t)}_{L^2}$ is uniformly
    bounded, and thus so is $\norm{\nabla \varphi_m (t)}_{L^2}$: for $
    t \in (- T_m, T^m)$, 
    \begin{align*}
        \norm{\nabla \varphi_m (t)}_{L^2} & \le 
\norm{\nabla u_0 }_{L^2} + \norm{\nabla u_m(t) }_{L^2}
\le \norm{\nabla u_0 }_{L^2} + \sqrt{\mathcal{E}_\textnormal{logGP}
 (u_m(t))} \\
&\le 2 \sqrt{\mathcal{E}_\textnormal{logGP} (u_0)}. 
    \end{align*}
    We know that $\nabla \varphi_m (t) = \sum_{k=0}^m g_{m, k} (t) \nabla w_k$, thus Lemma~\ref{lem:nabla_lin_indep} shows that we can apply Lemma \ref{lem:equiv_norm_finite_dim}, which gives
    \begin{equation*}
        \max_{0\le k\le m} \abs{g_{m,k} (t)} \leq C_m \norm{\nabla
          \varphi_m (t)}_{L^2} \le 2C_m \sqrt{\mathcal{E}_\textnormal{logGP} (u_0)} < \infty, \qquad \forall t \in (- T_m, T^m).
    \end{equation*}
    This proves that the $g_{m,k}$'s are actually globally bounded and therefore the solution is global.
\end{proof}

\subsection{Uniform estimates}

We have already shown the conservation of the energy and an estimate of the $L^2$-norm of $\nabla\varphi_m$, uniform both in $t$ and $m$. Now, we proceed with its $L^2$-norm.

\begin{lem} \label{lem:appr_L2_norm}
    Let $u_m = u_0 + \varphi_m$ the solution to \eqref{eq:eq_u_m}
    given by Lemma \ref{lem:exist_sol_app}. Then $\varphi_m$ is
    bounded in $\mathcal{C}^{0, \frac{1}{2}} (I, L^2)$ uniformly in
    $m$, for every bounded interval $I$. 
\end{lem}

\begin{proof}
    Since all $g_{m, k}$'s are continuous, we already know that
    $\varphi_m $ is continuous in time with values in   $L^2$.
    By multiplying \eqref{eq:eq_u_m} by $g_{m, j} (t)$ and summing over $j$, we obtain
    \begin{equation*}
        \< \varphi_m, i \partial_t \varphi_m + \Delta u_0 + \Delta
        \varphi_m - \lambda u_m \ln{\abs{u_m}^2} \>_{H^1, H^{-1}} = 0, 
    \end{equation*}
    and the imaginary part gives
    \begin{equation*}
        \frac{1}{2} \frac{\diff}{\diff t} \norm{\varphi_m}_{L^2}^2 =
        \Im \< \nabla \varphi_m, \nabla u_0 \> + \lambda \Im\<
        \varphi_m, u_m \ln{\abs{u_m}^2} \>, 
    \end{equation*}
    since $\Im \< \varphi_m, \Delta \varphi_m \>_{H^1, H^{-1}} =
    0$. Then, we can estimate the right-hand side : 
    \begin{equation*}
        \abs{\< \nabla \varphi_m, \nabla u_0 \>} \leq \norm{\nabla
          \varphi_m}_{L^2} \norm{\nabla u_0}_{L^2} \leq 2
        \mathcal{E}_\textnormal{logGP}(u_0), 
    \end{equation*}
    thanks to Lemma \ref{lem:appr_sol_global}, and
    \begin{align*}
        \abs{\< \varphi_m, u_m \ln{\abs{u_m}^2} \>}& \leq
\norm{\varphi_m}_{L^2}\norm{u_m \ln{\abs{u_m}^2}}_{L^2}
      \\
&\lesssim \norm{\varphi_m}_{L^2} \(\mathcal{E}_\textnormal{logGP}
           (u_0)^{1/2} + \mathcal{E}_\textnormal{logGP} (u_0)\), 
    \end{align*}
    by Corollary \ref{cor:v_ln_v_Lp}. Since $\varphi_m (0) = 0$, Gronwall lemma gives the uniform boundedness of the $L^2$ norm on every bounded interval $I$. Define
    \begin{equation*}
        \xi_m (t) \coloneqq -  \Delta \varphi_m -  \Delta u_0 + \lambda u_m \ln{\abs{u_m}^2}.
    \end{equation*}
    Then $\xi_m$ is bounded in $L^\infty (I, H^{-1})$ uniformly in $m$ for every bounded interval $I$. Moreover, for every $t, s \in I$,
    \begin{equation} \label{eq:int_diff_L2_norm}
        \norm{\varphi_m (t) - \varphi_m (s)}_{L^2}^2 = 2 \int_s^t \< \varphi_m (\tau) - \varphi_m (s), \partial_t \varphi_m (\tau) \>_{H^1, H^{-1}} \diff \tau.
    \end{equation}
    By multiplying again \eqref{eq:eq_u_m} at time $\tau$ by $g_{m, j} (\tau) - g_{m, j} (s)$ and summing over $j$, there holds
    \begin{equation*}
        \< \varphi_m (\tau) - \varphi_m (s), \partial_t \varphi_m (\tau) \>_{H^1, H^{-1}} = \< \varphi_m (\tau) - \varphi_m (s), \xi_m (\tau) \>_{H^1, H^{-1}}.
    \end{equation*}
    Since we know that the $H^1$-norms of $\varphi_m (\tau)$ and
    $\varphi_m (s)$ are bounded uniformly in $m$ for $\tau, s \in I$,
    we get that $\abs{\< \varphi_m (\tau) - \varphi_m (s), \xi_m
      (\tau) \>_{H^1, H^{-1}}}$ is bounded uniformly in $m \in
    \mathbb{N}$, for $s, \tau \in I$. Then, we get by \eqref{eq:int_diff_L2_norm}
    \begin{equation*}
        \norm{\varphi_m (t) - \varphi_m (s)}_{L^2}^2 \leq C_I \abs{t - s},
    \end{equation*}
    for all $t, s \in I$, with $C_I$ depending on $I$ but not on $m$, hence the conclusion.
\end{proof}

\subsection{Convergence}

\begin{lem} \label{lem:convergence}
    Let $u_m = u_0 + \varphi_m$ the solution to \eqref{eq:eq_u_m}
    given by Lemma~\ref{lem:exist_sol_app}. There exists a subsequence
    of $(\varphi_m)_m$ (still denoted by $\varphi_m$) and 
    \begin{equation*}
     \varphi \in
    \mathcal{C}^{0, 1} (\mathbb{R}, H^{-1}) \cap \mathcal{C}^{0,
      \frac{1}{2}} (\mathbb{R}, L^2) \cap L^\infty_\textnormal{loc}
    (\mathbb{R}, H^1) 
    \end{equation*}
such that $\varphi_m$ converges to $\varphi$ as
    $m \to \infty$ in the following sense: 
    \begin{itemize}
        \item $\varphi_m \wsconv \varphi$ in $L^\infty (I, H^1)$ for its weak-$*$ topology, for every bounded interval $I$,
        \item $\varphi_m (t) \rightharpoonup \varphi (t)$ in $L^2$ for its weak topology, for every $t \in \mathbb{R}$.
    \end{itemize}
    Moreover, $\lambda u_m \ln{\abs{u_m}^2}$ converges to $
    i\partial_t \varphi +  \Delta  \varphi + \Delta u_0 $ for the weak-$*$ topology of $L^\infty (\mathbb{R}, L^2)$.
\end{lem}

\begin{proof}
    We know by Lemmas \ref{lem:appr_sol_global} and
    \ref{lem:appr_L2_norm} that $\varphi_m$ is uniformly bounded in
    $L^\infty (I, H^1)$, which is the dual of $L^1 (I, H^{-1})$, for
    every bounded interval $I$. Then, this sequence is relatively
    compact for the weak-$*$ topology of $L^\infty (I, H^1)$. By diagonal
    extraction, there is a subsequence (still denoted by $\varphi_m$) which converges to some $\varphi \in L^\infty_\textnormal{loc} (\mathbb{R}, H^1)$ for the weak-$*$ topology of $L^\infty (I, H^1)$, for every bounded interval $I$.

    Besides, $u_m \ln{\abs{u_m}^2}$ is also uniformly bounded in $L^\infty (\mathbb{R}, L^2) = \Bigl( L^1 (\mathbb{R}, L^2) \Bigr)'$ thanks to Corollary \ref{cor:v_ln_v_Lp} and the conservation of the energy. Therefore, it is also relatively compact for the weak-$*$ topology of $L^\infty (\mathbb{R}, L^2)$. 

    Let $\theta \in \mathcal{C}_c^\infty (\mathbb{R})$ and  $j
    \le m$:
    \begin{align*}
        \lambda \int \theta (\tau) \< w_j, u_m \ln{\abs{u_m}^2} \>
      \diff \tau 
            &= \int \dot{\theta} (\tau) \< w_j, i \varphi_m (\tau) \>
              \diff \tau - \int \theta (t) \< \nabla w_j, \nabla
              \varphi_m (\tau) \> \diff \tau \\
&\quad - \int \theta (t) \<
              \nabla w_j, \nabla u_0 \> \diff \tau \\ 
            &\underset{m \rightarrow \infty}{\longrightarrow} \int
              \dot{\theta} (\tau) \< w_j, i \varphi (\tau) \> \diff
              \tau - \int \theta (t) \< \nabla w_j, \nabla \varphi
              (\tau) \> \diff \tau \\
&\qquad -\int \theta (t) \< \nabla w_j,
              \nabla u_0 \> \diff \tau. 
    \end{align*}
    Since $(w_j)_j$ is a Hilbert basis of $L^2$, this proves that
    $\lambda u_m \ln{\abs{u_m}^2}$ also converges to $i\partial_t
    \varphi + \Delta  \varphi + \Delta u_0$ in $\mathcal{D}'
    (\mathbb{R}, L^2)$. Therefore, the convergence is also in the
    weak-$*$ topology of $L^\infty (\mathbb{R}, L^2)$, thus we get
    $i\partial_t \varphi + \Delta  \varphi + \Delta u_0\in L^\infty
    (\mathbb{R}, L^2)$.  

    From this and the fact that $\Delta \varphi \in
    L^\infty_\textnormal{loc} (\mathbb{R}, H^{-1})$, we get
    $\partial_t \varphi \in L^\infty_\textnormal{loc} (\mathbb{R},
    H^{-1})$. Therefore, $\varphi \in W^{1, \infty}_\textnormal{loc}
    (\mathbb{R}, H^{-1})$ and thus $\varphi \in \mathcal C^{0,
      1}_\textnormal{loc} (\mathbb{R}, H^{-1}) \cap
    L^\infty_\textnormal{loc} (\mathbb{R}, H^1)$. By interpolation,
    this leads to $\varphi \in \mathcal C^{0, \frac{1}{2}}_\textnormal{loc}
    (\mathbb{R}, L^2)$. 

    As for the convergence in the second item, let $t \in
    \mathbb{R}$. Since $\varphi_m (t)$ is bounded in $L^2$ uniformly
    in $m$, the sequence is relatively compact for the weak topology
    of $L^2$. Let $\chi$ be the limit of a subsequence (still denoted
    by $\varphi_m$). It is then enough to prove that $\chi = \varphi
    (t)$. Let $\gamma \in (0, 1)$ Then, 
    \begin{align*}
        \norm{\varphi (t) - \chi}_{L^2}^2 &= \< \varphi (t) - \chi, \varphi_m (t) - \chi \> + \frac{1}{2 \gamma} \int_{t - \gamma}^{t + \gamma} \< \varphi (t) - \chi, \varphi (t) - \varphi_m (t) \> \diff \tau \\
            &= \< \varphi (t) - \chi, \varphi_m (t) - \chi \> \\
&+ \frac{1}{2 \gamma} \int_{t - \gamma}^{t + \gamma} \< \varphi (t) - \chi, \varphi (t) - \varphi (\tau) + \varphi (\tau) - \varphi_m (\tau) + \varphi_m (\tau) - \varphi_m (t) \> \diff \tau.
    \end{align*}
    Since $\varphi, \varphi_m \in \mathcal{C}^{0, \frac{1}{2}} ([t-1,
    t+1], L^2)$ uniformly in $m$, we have for all $\tau \in [t-1,
    t+1]$ 
    \begin{align*}
        \abs{\< \varphi (t) - \chi, \varphi (t) - \varphi (\tau) \>}
      &\leq C_t \abs{t - \tau}^\frac{1}{2} \norm{\varphi (t) -
        \chi}_{L^2}, \\ 
        \abs{\< \varphi (t) - \chi, \varphi_m (t) - \varphi_m (\tau)
      \>} &\leq C_t \abs{t - \tau}^\frac{1}{2} \norm{\varphi (t) -
            \chi}_{L^2}, \\ 
    \end{align*}
    hence
    \begin{equation}\label{eq:est_conv_phi_t} 
      \begin{aligned}
   \norm{\varphi (t) - \chi}_{L^2}^2 &\leq \< \varphi (t) - \chi,
  \varphi_m (t) - \chi \>+ \frac{C_t}{2 \gamma} \norm{\varphi (t) - \chi}_{L^2} \int_{t -
    \gamma}^{t+\gamma} \abs{t - \tau}^\frac{1}{2} \diff
   \tau \\
& \quad+\frac{1}{2 \gamma} \int_{t
    - \gamma}^{t + \gamma} \< \varphi (t) - \chi,
   \varphi (\tau) - \varphi_m (\tau) \> \diff \tau
     \\ 
   &\leq \< \varphi (t) - \chi, \varphi_m (t) - \chi \> + C_t \gamma^\frac{1}{2} \norm{\varphi (t) -
              \chi}_{L^2}\\
&\quad+ \frac{1}{2 \gamma} \int_{t - \gamma}^{t + \gamma} \<
              \varphi (t) - \chi, \varphi (\tau) - \varphi_m (\tau) \>
              \diff \tau .       
      \end{aligned}
    \end{equation}
    Moreover, we know that
    \begin{itemize}
        \item $\< \varphi (t) - \chi, \varphi_m (t) - \chi \> \underset{m \rightarrow \infty}{\longrightarrow} 0$ by weak convergence in $L^2$ of $\varphi_m (t)$ to $\chi$,
        \item $\int_{t - \gamma}^{t + \gamma} \< \varphi (t) - \chi, \varphi (\tau) - \varphi_m (\tau) \> \diff \tau \underset{m \rightarrow \infty}{\longrightarrow} 0$ by the weak-$*$ convergence in $L^\infty ((t - 1, t + 1), L^2)$ of $\varphi_m$ to $\varphi$.
    \end{itemize}
    Therefore, taking the limit $m \rightarrow \infty$ in \eqref{eq:est_conv_phi_t}, we get
    \begin{equation*}
        \norm{\varphi (t) - \chi}_{L^2}^2 \leq C_t \gamma^\frac{1}{2} \norm{\varphi (t) - \chi}_{L^2}.
    \end{equation*}
    Since this is true for all $\gamma \in (0, 1)$, we get
    $\norm{\varphi (t) - \chi}_{L^2}=0$, hence the conclusion.
\end{proof}

\subsection{Equation and initial data for the limit}

We are now able to prove that $u = u_0 + \varphi$ is indeed a
solution to \eqref{logGP}.

\begin{lem}
    Let $\varphi \in \mathcal{C}^{0, 1} (\mathbb{R}, H^{-1}) \cap
    \mathcal{C}^{0, \frac{1}{2}} (\mathbb{R}, L^2) \cap
    L^\infty_\textnormal{loc} (\mathbb{R}, H^1)$ defined in
    Lemma~\ref{lem:convergence}. Then $u \coloneqq u_0 + \varphi$
    is a weak solution to \eqref{logGP}. 
\end{lem}

\begin{proof}
    First, since $\varphi_m (0) \rightharpoonup \varphi (0)$ and
    $\varphi_m (0) = 0$ by construction, we get $\varphi (0) = 0$ and
    thus $u_{\mid t=0} = u_0$.

    We also know from Lemma~\ref{lem:convergence} that $\lambda u_m
    \ln{\abs{u_m}^2}$ converges to $i \partial_t \varphi +
     \Delta  \varphi + \Delta u_0 $ for the weak-$*$
    topology of $L^\infty (\mathbb{R}, L^2)$, and thus also for the
    weak topology of $L^2 (I \times \Omega)$ for every bounded
    interval $I$ and every bounded open subset $\Omega$ of $\mathbb{R}^d$. 

    We also know that $\varphi_m$ is $\frac{1}{2}$-H\"older continuous
    in time with values in $L^2 (\mathbb{R}^d)$ on $I$, uniformly in
    $m$ from Lemma \ref{lem:appr_L2_norm}, and thus also on $L^2
    (\Omega)$. Last, $\varphi_m (t) \in H^1$ for all $t \in I$ along
    with a uniform bound in $m$ by Lemmas \ref{lem:appr_sol_global}
    and \ref{lem:appr_L2_norm}. Thus, Arzela--Ascoli Theorem for
    ${\varphi_m}_{|\Omega}$ yields its relative compactness in
    $\mathcal{C} (I, L^2 (\Omega))$, which shows that
    ${\varphi_m}_{|\Omega}$ converges strongly to
    ${\varphi}_{|\Omega}$ in $\mathcal{C} (I, L^2 (\Omega))$, and thus
    also in $L^2 (I \times \Omega)$. Up to a further subsequence,
    $\varphi_m$ converges a.e. to $\varphi$ in $I \times \Omega$, and
    so does $u_m$ to $u$. By continuity of $z\mapsto z \ln{\abs{z}^2}$ on $\mathbb{C}$, we obtain the convergence of $u_m (t, x) \ln{\abs{u_m (t,x)}^2}$ to $u (t, x) \ln{\abs{u (t,x)}^2}$ a.e. in $I \times \Omega$.
    Besides, we also know that $u_m \ln{\abs{u_m}^2}$ is bounded in
    $L^\infty (I, L^2)$ uniformly in $m$, and thus also in $L^2 (I
    \times \Omega)$. Therefore, along this further subsequence, $u_m
    \ln{\abs{u_m}^2}$ converges weakly to $u \ln{\abs{u}^2}$ in $L^2 (I \times \Omega)$.

    Therefore, by uniqueness of the limit, we get $\lambda u
    \ln{\abs{u}^2} = i \partial_t \varphi +\Delta  \varphi + \Delta u_0 $, which gives the conclusion.
\end{proof}

\subsection{Conservation law}

We now prove that the energy of the solution that we have just
constructed is independent of time. Note that this is necessarily \emph{the}
solution to \eqref{logGP}, in view of Theorem~\ref{th:regu_sol}. 
\begin{lem}
    Let $\varphi \in \mathcal{C}^{0, 1} (\mathbb{R}, H^{-1}) \cap \mathcal{C}^{0, \frac{1}{2}} (\mathbb{R}, L^2) \cap L^\infty_\textnormal{loc} (\mathbb{R}, H^1)$ defined in Lemma~\ref{lem:convergence} and $u \coloneqq u_0 + \varphi$. Then $u$ satisfies $\mathcal{E}_\textnormal{logGP} (u (t)) = \mathcal{E}_\textnormal{logGP} (u_0)$ for all $t \in \mathbb{R}$.
\end{lem}

\color{black}
\begin{proof}
In view of the construction of the solution, Fatou's lemma yields
  \begin{equation*}
    \mathcal{E}_\textnormal{logGP} (u(t)) \le
    \mathcal{E}_\textnormal{logGP} (u_0),\quad \forall t\in \R.
  \end{equation*}
Arguing like in the proof of \cite[Theorem~3.3.9]{Cazenave_semlin_lognls}, for
$t_0\in \R$, we denote by $v$ the solution to
\eqref{logGP} with $v_{\mid t=0} = u_{\mid t=t_0}$.  We have like above,
 \begin{equation*}
    \mathcal{E}_\textnormal{logGP} (v(t)) \le
   \mathcal{E}_\textnormal{logGP} (v(0)) =
   \mathcal{E}_\textnormal{logGP} (u(t_0)),\quad \forall t\in \R. 
  \end{equation*}
By uniqueness (Theorem~\ref{th:regu_sol}), we infer
$v(t,x)=u(t+t_0,x)$, and, taking $t=-t_0$ in the previous inequality,
\begin{equation*}
  \mathcal{E}_\textnormal{logGP} (u_0) \le
   \mathcal{E}_\textnormal{logGP} (u(t_0)),\quad \forall t_0\in \R. 
\end{equation*}
We conclude 
  $\mathcal{E}_\textnormal{logGP} (u(t)) =
  \mathcal{E}_\textnormal{logGP} (u_0)$ for all $t\in \R$. 
\end{proof}

\color{black}


\section{Higher regularity}
\label{sec:higher}

In this section, we prove the propagation of the $\dot{H}^2$
regularity by the flow of \eqref{logGP}. We would like to proceed like
in \cite{carlesgallagher} (see also \cite{cazenave-haraux,
  Cazenave_semlin_lognls}) directly on the solution $u$ constructed,
i.e. differentiating \eqref{logGP} with respect to time and get an
$L^2$-energy estimate for $\d_t \varphi$. A direct formal computation
would lead to the local boundedness of this $L^2$-norm with the
additional argument that $\partial_t u (0) \in L^2$ as soon as $\Delta
u_0 \in L^2$. However, this computation cannot directly be made
rigorous since $u \ln{\abs{u}^2}$ cannot be differentiated due to the
lack of regularity of the logarithm. 

To overcome this difficulty, we shall work on the approximate
solutions $u_m$. In \cite{carlesgallagher}, the authors worked on the
approximate solutions to prove the propagation of the $H^2$
regularity, since the above mentioned energy estimate is licit in the
case of \eqref{eq:logNLSreg} (and provides bounds which are uniform in
$\eps$). However, we cannot reproduce the same proof here: our
approximate solutions do not satisfy an equation with a regularized
nonlinearity. Indeed, the equations \eqref{eq:eq_u_m} can be put under
the form 
\begin{equation*}
    i\partial_t \varphi_m + \mathbb{P}_m \Bigl( \Delta u_m - \lambda u_m \ln{\abs{u_m}^2} \Bigr) = 0,
\end{equation*}
where $\mathbb{P}_m$ is the orthogonal projector from $L^2$ onto $X_m$ (we can assume $w_j \in H^2$ to make this rigorous).

Thus, the nonlinearity is still not smooth enough to make the
computation rigorous. However, by assuming $w_j \in H^2$ (for all
$j$), we automatically have $\varphi_m (t) \in H^2$ for all $t \in
\mathbb{R}$ and $m \in \mathbb{N}$. Even more, we have $\varphi_m \in
\mathcal{C}^1 (\mathbb{R}, H^2)$ since the $g_{m, k}$'s belong to
$\mathcal{C}^1 (\mathbb{R})$ (see Lemma
\ref{lem:exist_sol_app}). Those crude bounds are \textit{a priori} not
uniform in $m$, but we show that we can improve them. 

\begin{lem} \label{lem:dt_um_unif_bound}
    Let $u_m$ given by Lemma~\ref{lem:exist_sol_app} and assume $w_j \in H^2$ for all $j \in \mathbb{N}$. Then, for all bounded interval $I$, $\partial_t u_m$ is bounded in $\mathcal{C}^0 (I, L^2)$ uniformly in $m$.
\end{lem}

\begin{proof}
    From Lemma \ref{lem:exist_sol_app}, we know that the $g_{m, j}$'s
    are $\mathcal{C}^{1, \varepsilon}$ for all $\varepsilon \in (0,
    1)$, therefore $\partial_t u_m \in\mathcal C^{0, \varepsilon} (\mathbb{R},
    H^2)$. Let $\tau > 0$ and 
    \begin{equation*}
        \psi_{m, \tau} (t) \coloneqq \frac{\norm{\partial_t \varphi_m (t + \tau)}_{L^2}^2 - \norm{\partial_t \varphi_m (t)}_{L^2}^2}{\tau},
    \end{equation*}
   which is well defined for all $t \in \mathbb{R}$ and any $\tau >
   0$. Our goal is to prove a bound on $\psi_{m, \tau} (t)$
   independent of $\tau \leq 1$. For this, we first rewrite 
    \begin{equation*}
        \psi_{m, \tau} (t) = \frac{1}{\tau} \Re \langle \partial_t \varphi_m (t + \tau) + \partial_t \varphi_m (t), \d_t\varphi_m (t + \tau) - \partial_t \varphi_m (t) \rangle.
    \end{equation*}
    Since $\partial_t \varphi_m (t+\tau)$ and $\partial_t \varphi_m (t)$ are in $\operatorname{Vect} (w_j)_{j \leq m}$, we can use \eqref{eq:eq_u_m} and get
    \begin{multline*}
        \tau \, \psi_{m, \tau} (t) = - \Im \langle \partial_t \varphi_m (t + \tau) + \partial_t \varphi_m (t), \Delta \varphi_m (t + \tau) - \Delta \varphi_m (t) \rangle \\ + \lambda \Im \langle \partial_t \varphi_m (t + \tau) + \partial_t \varphi_m (t), u_m (t+\tau) \ln{\abs{u_m (t+\tau)}^2} - u_m (t) \ln{\abs{u_m (t)}^2} \rangle.
    \end{multline*}

    For the first term, we get
    \begin{align*}
        \langle \partial_t \varphi_m (t), \Delta \varphi_m (t + \tau) - \Delta \varphi_m (t) \rangle &= \langle \partial_t \varphi_m (t), \Delta \varphi_m (t + \tau) - \Delta \varphi_m (t) - \tau \partial_t \Delta \varphi_m (t) \rangle \\&\quad+ \tau \langle \partial_t \varphi_m (t), \partial_t \Delta \varphi_m (t) \rangle.
    \end{align*}
    Since $\langle \partial_t \varphi_m (t), \partial_t \Delta \varphi_m (t) \rangle = - \norm{\partial_t \nabla \varphi_m}_{L^2}^2$, the last term vanishes when we take the imaginary part. On the other hand, we have shown that $\Delta \varphi_m \in \mathcal{C}^{1, \varepsilon} (\mathbb{R}, L^2)$. Hence, for all $\tau \leq 1$,
    \begin{equation*}
        \norm{ \Delta \varphi_m (t + \tau) - \Delta \varphi_m (t) - \tau \partial_t \Delta \varphi_m (t) }_{L^2} \leq C_{m, \varepsilon, t} \tau^{1 + \varepsilon}.
    \end{equation*}
    Therefore, for all $t \in \mathbb{R}$, $\tau, \varepsilon \in (0, 1)$,
    \begin{equation*}
        \abs{\Im \langle \partial_t \varphi_m (t), \Delta \varphi_m (t + \tau) - \Delta \varphi_m (t) \rangle} \leq C_{m, \varepsilon, t} \tau^{1 + \varepsilon},
    \end{equation*}
    and similarly for $\Im \langle \partial_t \varphi_m (t + \tau), \Delta \varphi_m (t + \tau) - \Delta \varphi_m (t) \rangle$.

    Now, we also have
    \begin{multline*}
        \langle \partial_t \varphi_m (t), u_m (t+\tau) \ln{\abs{u_m (t+\tau)}^2} - u_m (t) \ln{\abs{u_m (t)}^2} \rangle \\
            \begin{aligned}
                &= \frac{1}{\tau} \langle \varphi_m (t + \tau) - \varphi_m (t), u_m (t+\tau) \ln{\abs{u_m (t+\tau)}^2} - u_m (t) \ln{\abs{u_m (t+\tau)}^2} \rangle \\ &- \frac{1}{\tau} \langle \varphi_m (t + \tau) - \varphi_m (t) - \tau \partial_t \varphi_m (t), u_m (t+\tau) \ln{\abs{u_m (t+\tau)}^2} - u_m (t) \ln{\abs{u_m (t+\tau)}^2} \rangle.
            \end{aligned}
    \end{multline*}
    For the first term, we can use Lemma \ref{lem:log_inequality} since $\varphi_m (t + \tau) - \varphi_m (t) = u_m (t+\tau) - u_m (t)$, so that
    \begin{multline*}
        \abs{\Im \langle \varphi_m (t + \tau) - \varphi_m (t), u_m (t+\tau) \ln{\abs{u_m (t+\tau)}^2} - u_m (t) \ln{\abs{u_m (t+\tau)}^2} \rangle} \\ \leq 2 \norm{u_m (t+\tau) - u_m (t)}_{L^2}^2 = 2 \norm{\varphi_m (t+\tau) - \varphi_m (t)}_{L^2}^2.
    \end{multline*}
    For the second term, we can use Lemma \ref{lem:prop_lip_log} so that
    \begin{multline*}
        \abs{\langle \varphi_m (t + \tau) - \varphi_m (t) - \tau \partial_t \varphi_m (t), u_m (t+\tau) \ln{\abs{u_m (t+\tau)}^2} - u_m (t) \ln{\abs{u_m (t+\tau)}^2} \rangle} \\
        \leq \norm{\varphi_m (t+\tau) - \varphi_m (t) - \tau \partial_t \varphi_m (t)}_{L^2} \norm{u_m (t+\tau) \ln{\abs{u_m (t+\tau)}^2} - u_m (t) \ln{\abs{u_m (t+\tau)}^2}}_{L^2}
    \end{multline*}
    The first factor can be estimated like previously with the fact that $\varphi_m \in \mathcal{C}^{1, \varepsilon} (\mathbb{R}; L^2)$:
    \begin{equation*}
        \norm{\varphi_m (t+\tau) - \varphi_m (t) - \tau \partial_t \partial_t \varphi_m (t)}_{L^2} \leq C_{m, t, \varepsilon} \tau^{1 + \varepsilon}.
    \end{equation*}
    As for the other factor, we use Lemma \ref{lem:prop_lip_log} and then Corollary \ref{cor:v_ln_v_Lp} to obtain
    \begin{align*}
      &  \norm{u_m (t+\tau) \ln{\abs{u_m (t+\tau)}^2} - u_m (t) \ln{\abs{u_m (t+\tau)}^2}}_{L^2} \\
                \lesssim &
                    \norm{\Bigl( \abs{u_m (t +
                      \tau)}^\frac{\varepsilon}{2} \abs{\ln{\abs{u_m
                      (t + \tau)}}} + \abs{u_m
                      (t)}^\frac{\varepsilon}{2} \abs{\ln{\abs{u_m
                      (t)}}} \Bigr) \abs{u_m (t + \tau) - u_m (t)}^{1
                      - \frac{\varepsilon}{2}}}_{L^2} \\
      &+  \norm{u_m (t + \tau) - u_m (t)}_{L^2}\\
                 \lesssim &  \norm{\abs{u_m (t +
                       \tau)}^\frac{\varepsilon}{2} \abs{\ln{\abs{u_m
                       (t + \tau)}}} + \abs{u_m
                       (t)}^\frac{\varepsilon}{2} \abs{\ln{\abs{u_m
                       (t)}}}}_{L^\frac{4}{\varepsilon}}
                       \norm{\abs{u_m (t + \tau) - u_m (t)}^{1 -
                       \varepsilon}}_{L^\frac{2}{1 -
                       \frac{\varepsilon}{2}}} \\
       &+  \norm{u_m (t + \tau) - u_m (t)}_{L^2}\\
               \lesssim & 
                     \Bigl( \norm{\abs{u_m (t + \tau)}
                     \abs{\ln{\abs{u_m (t +
                     \tau)}}}^\frac{2}{\varepsilon}}_{L^2}^\varepsilon
                     + \norm{\abs{u_m (t)} \abs{\ln{\abs{u_m
                     (t)}}}^\frac{2}{\varepsilon}}_{L^2}^\eps \Bigr)
                     \norm{u_m (t + \tau) - u_m (t)}_{L^2}^{1 -
                     \frac{\varepsilon}{2}} \\
      &+ \norm{u_m (t + \tau) - u_m (t)}_{L^2}\\
               \lesssim &\Bigl( \mathcal{E}_\textnormal{logGP} (u_0)^{\eps/2} + \mathcal{E}_\textnormal{logGP} (u_0)^{(\eps + \eps^2)/2} \Bigr) \norm{u_m (t + \tau) - u_m (t)}_{L^2}^{1 - \frac{\varepsilon}{2}} + \norm{u_m (t + \tau) - u_m (t)}_{L^2}.
    \end{align*}
    Therefore, we obtain
    \begin{multline*}
        \abs{\langle \varphi_m (t + \tau) - \varphi_m (t) - \tau \partial_t \varphi_m (t), u_m (t+\tau) \ln{\abs{u_m (t+\tau)}^2} - u_m (t) \ln{\abs{u_m (t+\tau)}^2} \rangle} \\
        \leq C_{m, t, \varepsilon} \Bigl[ \tau^{2 +
          \frac{\varepsilon}{2}} \Bigl( \mathcal{E}_\textnormal{logGP}
        (u_0)^{\eps/2} +
        \mathcal{E}_\textnormal{logGP} (u_0)^{\eps} \Bigr)
        \norm{\delta_\tau \varphi_m}_{L^2}^{1 - \frac{\varepsilon}{2}}
        + \tau^{2+\varepsilon} \norm{\delta_\tau \varphi_m}_{L^2}
        \Bigr],
    \end{multline*}
    where $\delta_\tau \varphi_m \coloneqq \frac{1}{\tau} (\varphi_m (t+\tau) - \varphi_m(t))$. Thus, we get
    \begin{align*}
    &   \abs{\Im \langle \partial_t \varphi_m (t), u_m (t+\tau)
          \ln{\abs{u_m (t+\tau)}^2} - u_m (t) \ln{\abs{u_m (t)}^2}
      \rangle}  \\
   & \phantom{\Im \langle \partial_t \varphi_m (t)}   \le (2 \tau + C_{m, t, \varepsilon} \tau^{1 +
     \varepsilon}) \norm{\delta_\tau \varphi_m}_{L^2}^2 \\
      & \phantom{\Im \langle \partial_t \varphi_m (t)}  \quad+
        C_{m, t, \varepsilon} \tau^{1 + \frac{\varepsilon}{2}} \(
        \mathcal{E}_\textnormal{logGP} (u_0)^{\eps/2} +
        \mathcal{E}_\textnormal{logGP} (u_0)^{\eps} \)
        \norm{\delta_\tau \varphi_m}_{L^2}^{1 - \frac{\varepsilon}{2}} .
    \end{align*}
    Similar computations can be done for $\Im \langle \partial_t
    \varphi_m (t + \tau), u_m (t+\tau) \ln{\abs{u_m (t+\tau)}^2} - u_m
    (t) \ln{\abs{u_m (t)}^2} \rangle$, so we get
    \begin{equation} \label{eq:est_psi_m_tau}
      \begin{aligned}
               \abs{\psi_{m, \tau} (t)} &\leq C_{m, t, \varepsilon}
               \tau^\varepsilon + (4 \lambda + C_{m, t, \varepsilon}
               \tau^{\varepsilon}) \norm{\delta_\tau
                 \varphi_m}_{L^2}^2 \\
               &\quad + C_{m, t, \varepsilon} \tau^{1 +
                 \frac{\varepsilon}{2}} \Bigl(
               \mathcal{E}_\textnormal{logGP} (u_0)^{\eps/2} +
               \mathcal{E}_\textnormal{logGP} (u_0)^{\eps}
               \Bigr) \norm{\delta_\tau \varphi_m}_{L^2}^{1 -
                 \frac{\varepsilon}{2}}. 
      \end{aligned}
    \end{equation}

    At $m$ fixed, since $\varphi_m \in \mathcal{C}^{1, \varepsilon}(\R;L^2)$,
    we know that as $\tau\to 0$, $\delta_\tau \varphi_m (t)$ converges strongly to $\partial_t \varphi_m (t)$ in $L^2$ for any $t \in \mathbb{R}$, but also that, for $t \in \mathbb{R}$ fixed, $\norm{\delta_\tau \varphi_m}_{L^2}$ is uniformly bounded in $\tau \leq 1$. The previous estimate shows that $\psi_{m, \tau} (t)$ is uniformly bounded in $\tau \leq 1$ at $m \in \mathbb{N}$ and $t \in \mathbb{R}$ fixed, which means that
    \begin{equation*}
 t\mapsto        \norm{\partial_t \varphi_m (t)}_{L^2}^2 \in W^{1, \infty}_\textnormal{loc} (\mathbb{R}, \mathbb{R}).
    \end{equation*}
    The limit $\tau \rightarrow 0$ in \eqref{eq:est_psi_m_tau} then yields
    \begin{equation*}
        \abs{\frac{\diff}{\diff t} \norm{\partial_t \varphi_m (t)}_{L^2}^2} \leq 4 \lambda \norm{\partial_t \varphi_m (t)}_{L^2}^2.
    \end{equation*}
    On the other hand, we also have $i\partial_t \varphi_m (0) = \mathbb{P}_m \Bigl[ \lambda u_0 \ln{\abs{u_0}^2} - \Delta u_0 \Bigr]$, and since $\mathbb{P}_m$ is an orthogonal projector, we get
    \begin{align*}
        \norm{\partial_t \varphi_m (0)}_{L^2} &\le \lambda \norm{u_0
                                                \ln{\abs{u_0}^2}}_{L^2}
                                                + \norm{\Delta
                                                u_0}_{L^2} \\
      & \le C_\varepsilon \lambda (\mathcal{E}_\textnormal{logGP} (u_0)^{\eps/2} + \mathcal{E}_\textnormal{logGP} (u_0)^{\eps}) + \norm{\Delta u_0}_{L^2}.
    \end{align*}
    Thanks to Gronwall Lemma, we get a bound which is independent of $m$: for all $m \in \mathbb{N}$ and $t \in \mathbb{R}$,
    \begin{equation*}
        \norm{\partial_t \varphi_m (t)}_{L^2}^2 \leq \Bigl[
        C_\varepsilon \lambda (\mathcal{E}_\textnormal{logGP}
        (u_0)^{\eps/2} +
        \mathcal{E}_\textnormal{logGP} (u_0)^{\eps})
        + \norm{\Delta u_0}_{L^2} \Bigr]^2 e^{4
          \lambda t}.\qedhere
      \end{equation*}
 \end{proof}

\begin{cor}
    Let $u$ given by Lemma \ref{lem:convergence}. Then, for all bounded interval $I$, $\partial_t u$ is bounded in $L^\infty (I; L^2)$, and so is $\Delta u$.
\end{cor}

\begin{proof}
    We want to take the limit $m \rightarrow \infty$ in Lemma
    \ref{lem:dt_um_unif_bound}. For this, we use the fact that
    $\varphi_m \wsconv \varphi$ in $L^\infty (I; H^1)$ for its
    weak-$*$ topology for any bounded interval $I$ by Lemma
    \ref{lem:convergence}. Therefore, $\partial_t \varphi_m
    \rightharpoonup \partial_t \varphi$ in $\mathcal{D}' (I \times
    \mathbb{R}^d)$. As the $L^\infty (I; L^2)$ norm of $\partial_t
    \varphi_m$ is uniformly bounded in $m$ and since $\mathcal{C}_c (I
    \times \mathbb{R}^d)$ is dense in $L^\infty (I; L^2) = \Bigl( L^1
    (I;L^2) \Bigr)'$, this limit shows that $\partial_t \varphi =
    \partial_t u$ belongs to $L^\infty (I; L^2)$. 

    To conclude, we show that $u \ln{\abs{u}^2} \in L^\infty (I;
    L^2)$: in view of Corollary~\ref{cor:v_ln_v_Lp}, for $t\in I$, 
    \begin{align*}
        \norm{u (t) \ln{\abs{u (t)}^2} }_{L^2}  \lesssim
      \mathcal{E}_\textnormal{logGP} (u_0)^{1/2} +
      \mathcal{E}_\textnormal{logGP} (u_0) .
     \end{align*}
   Hence, using \eqref{logGP}, $\Delta u
    \in L^\infty (I; L^2)$. 
\end{proof}


\section{On stationary and traveling waves}
\label{sec:dynamics}

In this section, we prove some rather general results regarding 
solitary and  traveling waves. 
First, plugging \eqref{eq:trav_wave} into \eqref{logGP}, we get the equation for $\phi$:
\begin{equation} \label{eq:trav_phi}
    - \omega \phi - i c \cdot \nabla \phi + \Delta \phi = \lambda \phi \ln{\abs{\phi}^2}.
\end{equation}

\subsection{The only possible value for $\omega$}

\begin{proof}[Proof of Theorem \ref{th:value_omega}]
    $u$ satisfies \eqref{logGP} with initial data $\phi$ and $u \in
    L^\infty (\mathbb{R}; E_\textnormal{logGP})$. From Lemma
    \ref{lem:u_in_plusH1}, we thus know that $u - \phi \in
    \mathcal{C}^0 (\mathbb{R}; L^2)$. 
    On the other hand, we know that, for any $t \in \mathbb{R}$, $x \in \mathbb{R}^d$ and any $f$ smooth enough, 
    \begin{equation*}
        f (x - tc) - f(x) = - t \int_0^1 \nabla f (x - \tau c) \cdot c \diff \tau,
    \end{equation*}
    which leads to
    \begin{align*}
        \norm{f (x - tc) - f(x)}_{L^2} &\leq \abs{t} \int_0^1 \norm{\nabla f (x - \tau c) \cdot c}_{L^2} \diff \tau\leq \abs{t} \int_0^1 \norm{\nabla f (x - \tau c)}_{L^2} \abs{c} \diff \tau \\
            &\leq \abs{t} \abs{c} \int_0^1 \norm{\nabla f}_{L^2} \diff
              \tau \leq \abs{t} \abs{c} \norm{\nabla f}_{L^2}. 
    \end{align*}
    Since $\nabla \phi \in L^2$, we can also apply this inequality to $\phi$, which proves that for all $t \in \mathbb{R}$, $\phi (. - ct) - \phi \in L^2$.

   The above two claims imply 
    that $u -  \phi(\cdot -ct) \in L^2$ for all $t \in
    \mathbb{R}$. In view of \eqref{eq:trav_wave}, this means that $(e^{i \omega t} - 1) \phi \in L^2$
    for all $t \in \mathbb{R}$. Suppose that $\omega \neq 0$:
    considering $t = \frac{\pi}{\omega}$, we  get $- 2 \phi \in L^2$,
    which is in contradiction with the fact that $\abs{\phi} - 1 \in
    L^2$ from Lemma~\ref{lem:E_logGP_descr}. Therefore, $\omega=0$.
\end{proof}

\subsection{Regularity and limits of a traveling wave in dimension $d=1$}

In this section, we prove a result about the regularity and limits at
infinity of traveling waves of the form \eqref{eq:trav_wave} solution
to \eqref{logGP} in dimension $d=1$. In view of
Theorem~\ref{th:value_omega}, such traveling waves satisfy
\begin{equation} \label{eq:ode_phi_trav}
  -ic \phi' + \phi'' = \lambda \phi \ln{\abs{\phi}^2}.
\end{equation}

\begin{lem} \label{lem:solit_reg_lim}
    A traveling wave $\phi$ solution to \eqref{eq:ode_phi_trav} satisfies
    \begin{itemize}
        \item $\phi \in \mathcal{C}^2_b$ and $\lim_{\abs{x} \rightarrow \infty} \abs{\phi (x)} = 1$,
        \item $\phi' \in H^1$ and $\lim_{\abs{x} \rightarrow \infty} \phi' (x) = 0$.
    \end{itemize}
\end{lem}

\begin{proof}
    Since $\phi \in E_\textnormal{logGP}$, we have $\phi \in H^1_\textnormal{loc}$ and in particular $\phi \in \mathcal{C}^{0, \frac{1}{2}} (\mathbb{R})$ by Sobolev embedding in dimension $d=1$.
    Actually, $\phi$ satisfies \eqref{eq:ode_phi_trav} with $\omega = 0$ and $\phi \ln{\abs{\phi}^2} \in L^2 \cap \mathcal{C}^0$ from Corollary \ref{cor:v_ln_v_Lp}. Moreover, $\phi' \in L^2$ since $\phi \in E_\textnormal{logGP}$. Therefore, we get $\phi'' \in L^2$, which leads to $\phi' \in H^1$. By Sobolev embedding, this leads in particular to $\phi' \in \mathcal{C}^0_b$ and, coming back to \eqref{eq:ode_phi_trav} once again, we get $\phi'' \in \mathcal{C}^0$. This leads to $\phi \in H^2_\textnormal{loc} \cap \mathcal{C}^2$. $\phi' \in H^1$ also gives 
    $\lim_{\abs{x} \rightarrow \infty} \phi' (x) = 0$. Furthermore, with Lemma \ref{lem:E_logGP_descr}, $\abs{\phi} - 1 \in H^1$, which means that $\lim_{\abs{x} \rightarrow \infty} \abs{\phi (x)} = 1$ and that $\phi$ is bounded, which shows that $\phi''$ is also bounded by \eqref{eq:ode_phi_trav}.
\end{proof}


\section{Solitary waves in the one-dimensional case}
\label{sec:solitary}

In this section, we prove Theorem \ref{th:sol_wave} in the case $c=0$.
By Theorem~\ref{th:value_omega}, we know that $\omega = 0$. Thus, the
equation we address in this section is:
\begin{equation} \label{eq:solit_phi_ode}
    \phi'' = \lambda \phi \ln{\abs{\phi}^2}.
\end{equation}

\subsection{Properties of a solitary wave in dimension $d=1$}

We first assume that there exists a non-constant traveling wave $u$ of
the form \eqref{eq:trav_wave} with $c = 0$, solution to
\eqref{logGP} in the energy space, and we gather some of its properties. 

First, we show that $\phi$ satisfies an interesting energy equality.

\begin{lem} \label{lem:energy_ode}
    If $\phi\in E_\textnormal{logGP}$, then 
    \begin{equation} \label{eq:energy_ode}
        |\phi'|^2 = \lambda \( |\phi|^2 \ln{|\phi|^2} - |\phi|^2 + 1 \) .
    \end{equation}
\end{lem}

\begin{proof}
    Multiplying \eqref{eq:solit_phi_ode} by $2 \bar\phi' (x)$, and
    taking the real value, we get
    \begin{equation*}
        \frac{\diff}{\diff x} |\phi'|^2 = \lambda \frac{\diff}{\diff
          x} \( |\phi|^2        \ln{|\phi|^2} - |\phi|^2 + 1 \).
    \end{equation*}
    By the facts that $\phi \in E_\textnormal{logGP}$ and $|\phi|^2
    \ln{|\phi|^2} - |\phi|^2 + 1 \ge 0$, we get $|\phi'|^2 \in L^1$
    and $|\phi|^2 \ln{|\phi|^2} - |\phi|^2 + 1 \in L^1$, so we can
    integrate the above equation, and no additional constant of
    integration appears.
\end{proof}

\begin{cor} \label{cor:d_phi_sign}
    If $\phi\in E_\textnormal{logGP}$ is real-valued on some unbounded interval $I$, then either $\phi'$ never vanishes and does not change sign on $I$, or $\phi$ is constant on $I$ (equals to $\pm 1$).
\end{cor}

\begin{proof}
    If there is some $y \in I$ such that $\phi' (y) = 0$, then Lemma \ref{lem:energy_ode} shows that $\phi (y) = \pm 1$ as $x \mapsto x \ln{x} - x + 1$ vanishes only for $x = 1$ on $\mathbb{R}_+$, and thus the uniqueness part of the Cauchy theorem applied to \eqref{eq:solit_phi_ode} gives (as long as it does not vanish) $\phi \equiv \pm 1$ on $I$. On the other hand, if there is no such $y$, $\phi'$ cannot change sign on $I$ since it is continuous and never vanishes.
\end{proof}

Then, we show that $\phi$ can be taken real-valued up to a gauge on
some neighborhood of $+ \infty$. 

\begin{lem} \label{lem:infinite_int}
    There exists $x^- \in \mathbb{R} \cup \{ - \infty \}$ and $\theta
    \in [0, 2 \pi)$ such that $\phi (x) e^{-i \theta}$ is real-valued
    and positive for all $x \geq x^-$, and, either $x^- = - \infty$, or
    $\lim_{x\to x^-} \phi(x) = 0$. 
\end{lem}

\begin{proof}
    We can write \eqref{eq:solit_phi_ode} as a first order differential system,
    \begin{equation} \label{eq:ode_phi_psi}
        \begin{pmatrix}
            \phi \\
            \psi
        \end{pmatrix}'
        =
        \begin{pmatrix}
            \psi \\
            \lambda \phi \ln{\abs{\phi}^2}
        \end{pmatrix}
        \eqqcolon F \( 
        \begin{pmatrix}
            \phi \\
            \psi
        \end{pmatrix}
        \).
    \end{equation}
    $F$ is a continuous function on $\mathbb{C}^2$ and is of class
    $\mathcal{C}^1$ on $(\mathbb{C} \setminus \{ 0 \}) \times
    \mathbb{C}$. 
    From Lemma~\ref{lem:solit_reg_lim}, there exists $x_0 \in
    \mathbb{R}$ such that $\frac{3}{2}\ge \abs{\phi (x)} \ge \frac{1}{2}$ and
    $\abs{\phi' (x)} \leq \frac{1}{2}$ for all $x \geq
    x_0$. Therefore, by Cauchy-Lipschitz Theorem, there exists a maximal
    interval $I_+ = (x^-, x^+)$ such that $x_0 \in I_+$ and $(\phi,
    \phi')$ is the unique solution of \eqref{eq:ode_phi_psi} on $I_+$
    with initial data $\phi (x_0)$ and $\psi (x_0) = \phi' (x_0)$,
    with values in $\mathbb{C}^2 \setminus (\{ 0 \} \times
    \mathbb{C})$. Moreover, if $\abs{x^+} < \infty$, then either
    $\lim_{x \rightarrow x^+} \phi (x) = 0$ or $\lim_{x \rightarrow
      x^+} (\abs{\phi (x)} + \abs{\psi (x)}) = \infty$, and similarly
    for the limit $x \rightarrow x^-$. 

    However, since we already know that $\frac{3}{2}\ge \abs{\phi (x)}
    \ge  \frac{1}{2}$ and $\abs{\phi' (x)} \leq \frac{1}{2}$ for all
    $x \geq x_0$, we can already conclude that $x^+ = \infty$. 
    On the other hand, since $\phi \in \mathcal{C}^2 (\mathbb{R})$ by Lemma~\ref{lem:solit_reg_lim}, we know that $\lim_{x \rightarrow x^-} (\abs{\phi (x)} + \abs{\psi (x)}) = \infty$ cannot be true if $x^- > - \infty$.
    
    On $I_+$, $\phi$ does not vanish. Hence, we can use a polar decomposition: $\phi = \rho e^{i \theta}$, with $\rho$ and $\theta$ real-valued and defined as 
    \begin{equation} \label{eq:def_rho_theta}
        \rho = \abs{\phi}, \qquad
        \theta' = \frac{\Im ( \phi' \overline{\phi} )}{\rho^2}, \quad
        \theta (x_0) \in [0, 2 \pi) \text{ such that } e^{i \theta (x_0)} = \frac{\phi (x_0)}{\rho (x_0)}.
    \end{equation}
    The regularity of $\phi$ and the fact that it does not vanish on $I_+$ ensure that $\rho$ and $\theta$ are well defined and $\mathcal{C}^2$ on $I_+$. 
    Thus, we can compute $\phi''$ in terms of $\rho$, $\theta$, and
    their derivatives: 
    \begin{equation*}
        \phi'' = \(\rho'' - \rho (\theta')^2 + 2 i \rho' \theta'\) e^{i \theta}.
    \end{equation*}
    Substituting into \eqref{eq:solit_phi_ode}, we get (after simplification by $e^{i \theta}$)
    \begin{equation} \label{eq:polar_form_ode}
        \rho'' - \rho (\theta')^2 + 2 i \rho' \theta' + i \rho \theta'' = \lambda \rho \ln{\rho^2}.
    \end{equation}
    By taking the imaginary part, we get $\rho \theta'' + 2 \rho' \theta' = 0$. Since $\rho$ does not vanish on $I_+$, we get for all $x \in I_+$:
    \begin{equation*}
        \theta' (x) = \frac{c_0}{\rho (x)^2},
    \end{equation*}
    where $c_0$ is a constant.
    Then, we can compute for all $x \in I_+$
    \begin{equation*}
        \phi' (x) = \(\rho' (x) + i \rho (x) \theta' (x)\) e^{i \theta (x)} = \(\rho' (x) + i \frac{c_0}{\rho (x)}\) e^{i \theta (x)},
    \end{equation*}
    which leads to $\frac{c_0}{\rho} \in L^2 (I_+)$ and at the same time $\lim_{+ \infty} \frac{c_0}{\rho} = c_0$ since we know that $\lim_{x \rightarrow + \infty} \rho (x) = 1$. Combining these two arguments yields $c_0 = 0$.

    Hence, $\theta$ is constant on $I_+$, which gives the conclusion.
\end{proof}

From now on, without loss of generality, we can assume that $\theta =
0$ in Lemma~\ref{lem:infinite_int}, up to changing $\phi$ to $e^{- i
  \theta}\phi$.

\begin{lem} \label{lem:x_minus_prop}
    Let $x^- \in \mathbb{R} \cup \{ - \infty \}$ from
    Lemma~\ref{lem:infinite_int}. Then $x^- > - \infty$. Moreover,
    $\phi (x) \in (0, 1)$ and $\phi' (x) > 0$ for all $x > x^-$. Last,
    $\phi (x^-) = 0$ and $\phi' (x^-) = 1$. 
\end{lem}

\begin{proof}
    From Corollary \ref{cor:d_phi_sign}, we know that $\phi'$ never vanishes and does not change sign on $(x^-, \infty)$.
    By contradiction, if $\phi (x_2) > 1$ for some $x_2 > x^-$, then
    \begin{itemize}
        \item Either $\phi' > 0$ on $(x^-, \infty)$, and thus $\phi (x) \geq \phi (x_2)$ for all $x \geq x_2$, which is in contradiction with the fact that $\lim_\infty \abs{\phi} = 1$.
        \item Or $\phi' < 0$ on $(x^-, \infty)$, which means that $\phi (x) \geq \phi (x_2) > 1$ for all $x^- < x \leq x_2$. Therefore, $x^- = - \infty$, but then we have the same contradiction as in the previous case, for $- \infty$.
    \end{itemize}
    Therefore, $\phi (x) \in (0, 1)$ for all $x > x^-$, and this leads to $\phi' > 0$ on $(x^-, \infty)$.
    Now, remark that $f(x) = x^2 \ln{x^2} - x^2 + 1$ is decreasing on $[0, 1]$. Thus, $\lambda f(\phi) = (\phi')^2$ is decreasing, which shows that, taking $x_0 \in (x^-, \infty)$, we have $\phi' (x)^2 \geq \phi' (x_0)^2$ for all $x^- < x \leq x_0$, and thus $\phi' (x) \geq \phi' (x_0)$. By integration, we get for all $x \in (x^-, x_0]$,
    \begin{equation*}
        \phi (x) \leq \phi (x_0) + (x - x_0) \phi' (x_0).
    \end{equation*}
    Since the right hand side goes to $- \infty$ as $x \rightarrow -
    \infty$, and since $\phi > 0$ on $(x^-, \infty)$, this yields $x^- >
    - \infty$. Thus, we are in the second case of Lemma
    \ref{lem:infinite_int}, and since $\phi$ is continuous, we get
    $\phi (x^-) = 0$. Last, \eqref{eq:energy_ode} holds for $x = x^-$,
    and we know by continuity that $\phi' (x^-) \geq 0$, which gives
    $\phi' (x^-) = 1$. 
\end{proof}

\begin{lem} \label{lem:inifnite_int_2}
    Let $x^- \in \mathbb{R}$ from Lemma \ref{lem:infinite_int}. Then, for all $x < x^-$, $\phi (x)$ is real-valued, negative and increasing.
\end{lem}

\begin{proof}
    From Lemma \ref{lem:x_minus_prop}, we know that $\phi (x^-) = 0$
    and $\phi' (x^-) = 1$. Therefore, there exists $\delta > 0$ such
    that $\phi \neq 0$ on $[x^- - \delta, x^-)$. Define 
    \begin{equation*}
        x^-_0 \coloneqq \inf \{ y < x^- \, | \, \forall x \in (y, x^-), \phi (x) \neq 0 \} < x^- - \delta.
    \end{equation*}
   On $I_2 \coloneqq (x^-_0, x^-)$, we can use a polar factorization in
   the same way as in \eqref{eq:def_rho_theta}: $\phi = \rho e^{i
     \theta}$,  where $\rho$ and $\theta$ also satisfy
   \eqref{eq:polar_form_ode} on $I_2$. Taking again the imaginary
   part and integrating the ODE, we have, like before,
    \begin{equation*}
        \theta' (x) = \frac{c_0}{\rho (x)^2}, \qquad \forall x \in I_2.
    \end{equation*}
    We still have then
    \begin{equation*}
        \phi' (x) = \(\rho' (x) + i \frac{c_0}{\rho (x)}\) e^{i \theta (x)},
    \end{equation*}
    and once again we must have $\frac{c_0}{\rho (x)} \in L^2 (I_2)$
    since $\phi' \in L^2$. However, we know that $\phi (x^-) = 0$ and
    $\phi' (x^-) = 1$, which means that $\phi (x) \sim x - x^-$ as $x\to x^-$.
    In terms of $\rho$, this implies $\rho (x) \sim\abs{x - x^-}$
    as $x\to x^-$, and $\frac{c_0}{\rho (x)} \in 
    L^2 (I_2)$ if and only if $c_0 = 0$. Once again, $\theta$ is
    constant on $I_2$, and therefore so is $\frac{\phi}{\rho}$. With
    the previous asymptotics, we know that it tends to $-1$ at
    $x^-$. Thus, we get 
    \begin{equation*}
        \phi = - \rho.
    \end{equation*}
    This shows that $\phi$ is real-valued and negative on $I_2$. Then, we can apply Lemma~\ref{lem:energy_ode} on $(x^-_0, \infty)$, which shows that $\phi'$ does not vanish on this interval, and thus $\phi'$ does not change sign once again, which proves that $\phi' > 0$ on $I_2$. Therefore, $\phi (x) \leq \phi (x^- - \delta) < 0$ for all $x \in (x^-_0, x^- - \delta)$, which proves that $x^-_0 = - \infty$.
\end{proof}

It only remains to analyze the limit of $\phi$ at $- \infty$.

\begin{lem}
  We have
  \begin{equation*}
    \lim_{x\to - \infty} \phi    (x)= - 1.
  \end{equation*}
\end{lem}

\begin{proof}
 We have seen that   $\phi$ is increasing, thus it has a limit at $-
 \infty$. From Lemma~\ref{lem:inifnite_int_2}, we know it is
 negative. Moreover, $\abs{\phi} - 1 \in L^2$, which means that $\phi
 + 1 \in L^2 ((- \infty, 0))$. The conclusion easily follows. 
\end{proof}

The conclusion of all the previous lemmas is the following:

\begin{cor} \label{cor:necessary_cond}
    If $u$ is a non-constant traveling wave  solution to \eqref{logGP}
    of the form \eqref{eq:trav_wave} with $c=0$, then there exists $\theta \in \mathbb{R}$ such that $e^{- i \theta} \phi$ is a real-valued,  increasing, $\mathcal{C}^2$ function with values in $(-1, 1)$ which vanishes at a unique point $x_0$.
    Moreover, $\phi_0 (x) = e^{- i \theta} \phi (x + x_0)$ also
    satisfies \eqref{logGP},
    \begin{equation*}
        \lim_{x \rightarrow \pm \infty} \phi_0 (x) = \pm 1,
    \end{equation*}
    and
    \begin{equation} \label{eq:energy_ode_compl}
        (\phi_0')^2 = \lambda \Bigl( \phi_0^2 \ln{\phi_0^2} - \phi_0^2 + 1 \Bigr) \qquad \text{ on } \mathbb{R}.
    \end{equation}
\end{cor}

\subsection{Analysis of the new ODE and proof of Theorem \ref{th:sol_wave}}

From Corollary~\ref{cor:necessary_cond}, $\phi_0$ vanishes at $x=0$,
satisfies \eqref{eq:energy_ode_compl} and is strictly increasing,
i.e. $\phi_0' > 0$. Thus, $\phi_0$ satisfies the ODE 
\begin{equation} \label{eq:new_ode}
    \phi_0' = \sqrt{\lambda} \sqrt{\phi_0^2 \ln{\phi_0^2} - \phi_0^2 + 1} \qquad \text{ on } \mathbb{R}.
\end{equation}
The uniqueness of this function is the topic of the following lemma.

\begin{lem} \label{lem:exist_unique_new_ode}
    There exists a unique function $\phi_0$ satisfying \eqref{eq:new_ode} with the initial data $\phi_0 (0) = 0$. It is defined on $\mathbb{R}$ and satisfies
    \begin{equation*}
        \lim_{x \rightarrow \pm \infty} \phi_0 (x) = \pm 1.
    \end{equation*}
    Moreover, $\phi_0 \in E_\textnormal{logGP}$.
\end{lem}

\begin{proof}
    We already know that $f(x) = x \ln{x} - x + 1 \geq 0$ for all $x
    \geq 0$, and $f$ vanishes only at $x=1$. Moreover, by simple
    computations, we check that $g(x) = f(x^2)$ is $\mathcal{C}^1$ on
    $\mathbb{R}$, and $\mathcal C^2$ on $\R\setminus\{0\}$. We 
    also compute $f' (1) = 0$ and $f'' (1) = 1$, hence
    \begin{equation*}
      g(x)\Eq x {\pm 1} \frac{1}{2}\(x^2-1\)^2.
    \end{equation*}

    From these facts, we deduce that $h = \sqrt{g}$ is a
    $\mathcal{C}^0 (\mathbb{R}) \cap \mathcal{C}^1 (\mathbb{R}
    \setminus \{ -1, 1 \})$ function such that $h(x) \sim_{x
      \rightarrow \pm 1} \abs{x^2 - 1}/\sqrt 2$. Therefore, it is
    $\mathcal{C}^{0, 1}$ locally on $\mathbb{R}$, and we may invoke
    Cauchy-Lipschitz Theorem: there exists a unique $\phi_0$ satisfying
    $\phi_0' = \sqrt{\lambda} h(\phi_0)$ with the initial condition
    $\phi_0 (0) = 0$, on a maximal interval $I\ni 0$ of existence. 

    Then, $\phi_0$ cannot reach $1$ or $-1$, because the constant
    function $1$ and $-1$ are both solutions to this ODE, and it would
    contradict the uniqueness in the Cauchy-Lipschitz Theorem. Since $\phi_0$ is
    continuous, we infer that $\phi_0 (x) \in (-1, 1)$ for all $x
    \in \mathbb{R}$. Moreover, we can easily prove that there exists
    $c_1> 0$ such that 
    \begin{align*}
        h(y) &\ge c_1 (y + 1), \quad\forall y \in [-1, 0], \\
        h(y) &\ge c_1 (1 - y),\quad \forall y \in [0, 1].
    \end{align*}
    Since we know that $\phi_0 (x) \in (-1, 0]$ for all $x \leq 0$, we can estimate
    \begin{equation*}
        \phi_0' (x) \geq c_1 \sqrt{\lambda} (\phi_0 (x) + 1), \qquad \forall x \leq 0.
    \end{equation*}
    By integrating backward, we get for all $x \leq 0$,
    \begin{align*}
        -1 < \phi_0 (x) &\leq -1 + e^{\sqrt{\lambda} c_1 x}, \\
        0 < \phi_0' (x) &\leq c_1 \sqrt{\lambda} e^{\sqrt{\lambda} c_1 x}.
    \end{align*}
    Similarly, we also have for all $x \geq 0$
    \begin{align*}
        1 - e^{- \sqrt{\lambda} c_1 x} &\leq \phi_0 (x) < -1, \\
        0 < \phi_0' (x) &\leq c_1 \sqrt{\lambda} e^{\sqrt{\lambda} c_1 x}.
    \end{align*}
    Those estimates prove the expected limits at $\pm \infty$, and also
    show that $\phi_0' \in L^2$ together with $\abs{\phi_0} - 1 \in L^2$. We
    conclude that $\phi_0 \in E_\textnormal{logGP}$ by
    Lemma~\ref{lem:E_logGP_descr}.  
\end{proof}

The previous result obviously proves that we have a set of solitary
waves for \eqref{logGP}. 

\begin{cor} \label{cor:set_soli_waves}
    For any $\theta \in \mathbb{R}$ and any $x_0 \in \mathbb{R}$,
    $e^{i \theta} \phi_0 (. - x_0) \in E_\textnormal{logGP}$ is a
    solitary (stationary) wave of \eqref{logGP}.
\end{cor}

It remains to show that they are the only possible solitary waves in
order to prove Theorem~\ref{th:sol_wave} in the case $c=0$.

\begin{proof}[Proof of Theorem \ref{th:sol_wave}: case $c=0$.]
    Let $\phi \in E_\textnormal{logGP}$ and $u$ be a traveling wave of
    the form \eqref{eq:trav_wave} with $c=0$. From Corollary
    \ref{cor:necessary_cond}, we can find $\theta \in \mathbb{R}$ and
    $x_0 \in \mathbb{R}$ such that $e^{- i \theta} \phi (. + x_0)$ is
    a real-valued, increasing, $\mathcal{C}^2$ function
    with values in $(-1, 1)$, which vanishes at $0$, and satisfies
    \eqref{eq:energy_ode_compl}, i.e. \eqref{eq:new_ode} by
    positivity of its derivative. 
    It is therefore $\phi_0$, defined in
    Lemma~\ref{lem:exist_unique_new_ode} by uniqueness given by the
    same lemma: $e^{- i \theta} \phi (. + x_0) = \phi_0$, i.e. $\phi =
    e^{i \theta} \phi_0 (. - x_0)$. 
\end{proof}


\section{Traveling waves in the one-dimensional case}
\label{sec:traveling}

\subsection{Admissible velocities}
\label{sec:velocities}

The goal of this section is to prove Theorem~\ref{th:trav_waves_1d},
and characterize velocities of nontrivial (that is, nonconstant)
traveling waves. We assume that there exists a traveling
wave of the form \eqref{eq:trav_wave} solution to \eqref{logGP},
i.e. $\phi \in E_\textnormal{logGP}$ solution to
\eqref{eq:ode_phi_trav}. 

\begin{lem} \label{lem:another_energy_eq}
 The function   $\eta = 1 - \abs{\phi}^2$ satisfies, on $\mathbb{R}$,
    \begin{equation*}
        \frac{1}{2} (\eta')^2 = h_c (\eta),
      \end{equation*}
      where
\begin{equation*}
    h_c (y) \coloneqq \lambda \Bigl( (1-y)^2 \ln(1-y)^2 - (1-y)^2 + 1 \Bigr) - \frac{2 \lambda +c^2}{2} y^2.
\end{equation*}
\end{lem}

\begin{proof}
 We follow the same lines as in \cite{BeGrSa08}.   Writing
 $\phi=\phi_1+i \phi_2$, \eqref{eq:ode_phi_trav} becomes 
    \begin{System} \label{eq:vR}
        \phi_1''+ c \phi_2' = \lambda \phi_1\ln\(\phi_1^2+\phi_2^2\),\\[5pt]
        \phi_2''- c \phi_1' = \lambda \phi_2\ln\(\phi_1^2 + \phi_2^2\).
    \end{System}
    We multiply the first equation by $\phi_2$ and the second one by $\phi_1$, and we subtract in order to get:
    \begin{equation*}
        \( \phi_1' \phi_2 - \phi_1 \phi_2'\)' = -c (\phi_2 \phi_2' + \phi_1 \phi_1') = \frac{c}{2} \eta', \quad
        \eta \coloneqq 1-\abs{\phi}^2.
    \end{equation*}
    Moreover, we know that $\phi$ is bounded by
    Lemma~\ref{lem:solit_reg_lim}, thus $\phi_1' \phi_2, \phi_1
    \phi_2' \in L^2$, and $\eta \in L^2$ by using
    Lemma~\ref{lem:E_logGP_descr}. Thus,  integrating the above
    identity yields
    \begin{equation}\label{eq:wronskien}
       \phi_1' \phi_2 - \phi_1 \phi_2' = \frac{c}{2} \eta.
    \end{equation}

    Now, we multiply the first equation of \eqref{eq:vR} by $\phi_1'$, the second by $\phi_2'$ and we sum in order to get:
    \begin{equation*}
      \phi_1' \phi_1'' + \phi_2' \phi_2'' = \lambda \(\phi_1 \phi_1' + \phi_2 \phi_2'\)\ln \(\phi_1^2 + \phi_2^2\),
    \end{equation*}
    which can be written as
    \begin{equation*}
        \(\abs{\phi'}^2\)' = \lambda \(\abs{\phi}^2 \ln{\abs{\phi}^2} -
        \abs{\phi}^2 + 1\)'. 
    \end{equation*}
    We know that both $\abs{\phi'}^2$ and $\abs{\phi}^2 \ln{\abs{\phi}^2} - \abs{\phi}^2 + 1$ (which is positive and whose integral over $\mathbb{R}$ is bounded by $\mathcal{E}_\textnormal{logGP} (\phi)$) are in $L^1$, hence:
    \begin{equation*}
        \abs{\phi'}^2 = \lambda \bigl( \abs{\phi}^2 \ln{\abs{\phi}^2} - \abs{\phi}^2 + 1 \Bigr) \qquad \text{on } \mathbb{R}.
    \end{equation*}
    We deduce then
    \begin{align*}
        \eta'' &= -2|\phi'|^2 - 2 (\phi_1 \phi_1'' + \phi_2 \phi_2'')\\
            &= -2|\phi'|^2 -2 \phi_1 \(-c \phi_2' + \lambda \phi_1 \ln{\abs{\phi}^2}\) - 2 \phi_2 \( c \phi_1' + \lambda \phi_2 \ln{\abs{\phi}^2}\)\\
            &=-4 \lambda (1-\eta)\ln(1-\eta) - 2 \lambda \eta - c^2\eta = h_c' (\eta).
    \end{align*}
    Multiplying by $\eta'$, we obtain
    \begin{equation*}
        \frac{1}{2} \( (\eta')^2 \)' = ( h_c (\eta) )'. 
    \end{equation*}
    Since $\lim_{\pm \infty} \eta' = \lim_{\pm \infty} \eta = 0$ from
    Lemma~\ref{lem:solit_reg_lim}, we conclude by integrating.
\end{proof}

\begin{rem}
    This lemma may be extended to the case of any non-linearity $F(u)$
    with $F$ a continuous function on $\mathbb{C}$ of the form $F (u)
    = u \, f(\abs{u}^2)$ with $f(1) = 0$. 
\end{rem}

\begin{cor} \label{cor:trav_wave_never_vanish}
    If $c \neq 0$, then $\phi$ never vanishes.
\end{cor}

\begin{proof}
    It comes from the fact that $h_c (1) = - \frac{c^2}{2} < 0$ if $c
    \neq 0$, while $h_c (\eta) = \frac{1}{2} (\eta')^2 \geq
    0$. Therefore, $\eta (x) \neq 1$, that is $\phi (x) \neq 0$, for
    all $x \in \mathbb{R}$. 
\end{proof}

\begin{cor} \label{cor:eta_zero_large_c}
    If $c^2 \ge 2 \lambda$, then $\eta \equiv 0$.
\end{cor}

\begin{proof}
    The limiting case $c^2=2\lambda$ is treated like in
    \cite{BeGrSa08}: if $\eta$ is not identically zero, then in view
    of Lemma~\ref{lem:solit_reg_lim}, by
    translation invariance, we may assume
    \begin{equation*}
      |\eta(0)|=\max\{|\eta(x)|,\ x\in \R\}>0,\quad \eta'(0)=0.
    \end{equation*}
In view of Lemma~\ref{lem:another_energy_eq}, $h_c(\eta(0))=0$. On the
other hand, Corollary~\ref{cor:trav_wave_never_vanish} implies that
$1-\eta(x)>0$ for all $x\in \R$, and direct computation yields
\begin{equation*}
  h''_c(y) =4\lambda\ln(1-y),\quad 0\le y<1.
\end{equation*}
Therefore, $h''_c(y)$ is positive for $y<0$, negative for $0<y<1$. As
$h'_c(0)=0$, $h'_c$ is negative on $(-\infty,0)\cup (0,1)$: $h_c$ is
decreasing on $(-\infty,1)$.
Since $h_c(0)=0$, the only zero of $h_c$ on $[0,1)$ is the origin, and so
$\eta(0)=0$, hence a contradiction: $\eta$ is identically zero.
\smallbreak

    Suppose now $c^2>2\lambda$. Since $h_c(0)=h_c'(0)=0$, and
    $h_c''(0)= 2\lambda-c^2$,
    there exists $\varepsilon > 0$ such that $h_c (y) < 0$ for all $y
    \in [- \varepsilon, \varepsilon] \setminus \{ 0 \}$. On the other
    hand, $h_c (\eta)$ is  continuous, nonnegative (from
    Lemma~\ref{lem:another_energy_eq}) and tends to $0$ at $\pm
    \infty$. The  conclusion follows easily. 
\end{proof}

\begin{lem}
    If $c^2 \ge 2 \lambda$, then $\phi$ is constant.
\end{lem}

\begin{proof}
    By Corollary \ref{cor:eta_zero_large_c}, we know that $\abs{\phi
      (x)} = 1$ for all $x \in \mathbb{R}$. Since $\phi$ is a
    $\mathcal{C}^2_b (\mathbb{R})$ function, there exists a
    real-valued function $\theta \in \mathcal{C}^2$ such that $\phi =
    e^{i \theta}$ (defined like in \eqref{eq:def_rho_theta} for
    instance). 
    By substitution in \eqref{eq:ode_phi_trav}, $\theta$ satisfies then
    \begin{equation*}
        c \theta' + i \theta'' - (\theta')^2 = 0.
    \end{equation*}
    By taking the imaginary part, we get $\theta'' = 0$. Since we must
    have $\phi' = i \theta' e^{i \theta} \in L^2$, we get $\theta' \in
    L^2$ and thus $\theta' = 0$. 
\end{proof}

\subsection{Nontrivial traveling waves}

We now consider the case $0 < c^2 < 2 \lambda$.
Define, for all $y > 0$,
\begin{align*}
    f_c (y) &\coloneqq \frac{c^2}{4} \Bigl( \frac{1}{y^3} - y \Bigr) + \lambda y \ln{y^2}, \\
    g_c (y) &\coloneqq - \frac{c^2}{4} \frac{(1 - y^2)^2}{y^2} +
              \lambda (y^2 \ln{y^2} - y^2 + 1). 
\end{align*}
The following lemma is established by direct calculations:
\begin{lem} \label{lem:analysis_fcts}
    Let $c$ such that $0 < c^2 < 2 \lambda$. There exists $0<y_0<y_1<1$
    such that the following holds:
    \begin{itemize}
\item $f_c$ has exactly two zeroes on $(0,+\infty)$: $y_1$ and $1$. 
    \item $f_c$ is positive on $(0,y_1)\cup (1,+\infty)$, negative on
      $(y_1,1)$. 
\item $g_c$ has exactly two zeroes on $(0,+\infty)$: $y_0$ and $1$.
\item $g_c$ is negative on $(0,y_0)$, positive on
  $(y_0,+\infty)\setminus \{1\}$.
\item There exists $C_c>0$ such that $\frac{1}{C_c} (1 - y)^2 \leq g_c
  (y) \leq C_c (1 - y)^2$ for all $y \in (y_1, 1)$.
    \end{itemize}
\end{lem}

\begin{proof}
  We compute
  \begin{equation*}
    f_c'(y) = \frac{c^2}{4}\(-\frac{3}{y^4}-1\)+2\lambda \ln y
    +2\lambda,\quad f_c''(y) = \frac{3c^2}{y^5}+\frac{2\lambda}{y}>0. 
  \end{equation*}
  As
  \begin{equation*}
    \lim_{y\to 0} f_c'(y) = -\infty,\quad \lim_{y\to +\infty} f_c'(y) = +\infty,
  \end{equation*}
  the derivative $f_c'$ has a unique zero on $(0,+\infty)$, which we
  denote by $y_2>0$. Since $f_c'(1) = 2\lambda-c^2>0$, we know that
  $0<y_2<1$: $f_c(1)=0$, hence $f_c(y_2)<0$, and there exists a unique
  $0<y_1<y_2$ such that $f_c(y_1)=0$. 
  \smallbreak

We note that  $f_c=\frac{1}{2}g_c'$, and in view of the above pieces
of information, we can draw:
\smallbreak

\begin{tikzpicture}
\tkzTabInit[lgt=1.2,espcl=2]{$y$ / 1 , $f_c(y)$ / 1 , $g_c(y)$ / 1}{0,
  $y_1$, $y_2$, $1$, $+\infty$} 
\tkzTabVar{D+ /  $+\infty$ , R/ , - / $-$, R/ , + / $+\infty$ }
\tkzTabIma{1}{3}{2}{$0$}
\tkzTabIma{3}{5}{4}{$0$}
\tkzTabVar{D- / $-\infty$ , + / $+$ ,  R/ , - / $0$ , + / $+\infty$ }
\end{tikzpicture}

In particular, $g_c$ is increasing on $(0,y_1)$, from $-\infty$ to a
positive value, and there exists a unique $y_0\in (0,y_1)$ such that
$g_c(y_0)=0$, and the lemma follows easily.   
\end{proof}

\begin{lem} \label{lem:carac_trav_waves}
    If $c$ is such that $0 < c^2 < 2 \lambda$, then $\rho = \abs{\phi}$ satisfies
    \begin{equation} \label{eq:rho_ode}
        \rho'' = f_c (\rho),
    \end{equation}
    \begin{equation} \label{eq:energy_ode_c_compl}
        (\rho')^2 = g_c (\rho).
    \end{equation}
    Moreover, there exists $x_0 \in \mathbb{R}$ such that $\rho (x_0)
    = y_0$, where $y_0$ is defined in Lemma~\ref{lem:analysis_fcts},
    and $\theta \in \mathcal{C}^2$, defined so that $\phi = \rho e^{i
      \theta}$, satisfies 
    \begin{equation} \label{eq:theta_ode}
        \theta' = \frac{c}{2} \Bigl( 1 - \frac{1}{\rho^2} \Bigr).
    \end{equation}
\end{lem}

\begin{proof}
    By Corollary \ref{cor:trav_wave_never_vanish}, $\phi$ never vanishes. It is also a $\mathcal{C}^2_b$ function. Therefore, we can define $\rho$ and $\theta$ as in \eqref{eq:def_rho_theta} so that $\phi = \rho e^{i \theta}$. They satisfy
    \begin{equation} \label{eq:trav_ode_theta_rho}
        - i c \rho' + c \rho \theta' + \rho'' - \rho (\theta')^2 + 2 i \rho' \theta' + i \rho \theta'' = \lambda \rho \ln{\rho^2}.
    \end{equation}
    By taking the imaginary part, we get
    \begin{equation} \label{eq:ode_theta}
        - c \rho' + \rho \theta'' + 2 \rho' \theta' = 0.
    \end{equation}
    As $\rho$ never vanishes, the solution of this ODE on $\theta'$ takes the form
    \begin{equation*}
        \theta' (x) = \frac{c}{2} + \frac{c_0}{\rho (x)^2},
    \end{equation*}
    where $c_0$ is a real constant. Then, there holds
    \begin{equation*}
        \phi' = \Bigl( \rho' + i \Bigl( \frac{c}{2} \rho + \frac{c_0}{\rho} \Bigr) \Bigr) e^{i \theta} \in L^2.
    \end{equation*}
    On the other hand, we already know that $\lim_{\pm \infty} \rho = 1$, thus we must have $c_0 = - \frac{c}{2}$, so that
    \begin{equation*}
        \theta' (x) = \frac{c}{2} \Bigl( 1 - \frac{1}{\rho (x)^2} \Bigr).
    \end{equation*}
    Now, \eqref{eq:trav_ode_theta_rho} reads
    \begin{align*}
        \rho'' &= - \frac{c^2}{2} \Bigl( \rho - \frac{1}{\rho} \Bigr) + \frac{c^2}{4} \rho \Bigl( 1 - \frac{1}{\rho^2} \Bigr)^2 + \lambda \rho \ln{\rho^2} \\
            &= - \frac{c^2}{2} \Bigl( \rho - \frac{1}{\rho} \Bigr) + \frac{c^2}{4} \Bigl( \rho - \frac{2}{\rho} + \frac{1}{\rho^3} \Bigr) + \lambda \rho \ln{\rho^2} \\
        &= \frac{c^2}{4} \Bigl( \frac{1}{\rho^3} - \rho \Bigr) + \lambda \rho \ln{\rho^2}.
    \end{align*}
    %
  After multiplication by $\rho'$ and integration, we find:
    \begin{equation*}
        (\rho')^2 = - \frac{c^2}{4} \frac{(1 - \rho^2)^2}{\rho^2} + \lambda (\rho^2 \ln{\rho^2} - \rho^2 + 1).
    \end{equation*}
    We also know that if $\rho \equiv 1$, then $\theta' \equiv 0$ and
    thus $\phi$ is constant. On the contrary, if $\phi$ is not
    constant, then $\rho$ is not either, and thus has an extremum since
    $\lim_{\pm \infty} \rho = 1$. Let $x_0$ be extremal. Then $\rho'
    (x_0) = 0$, and therefore $0 = g_c 
    (\rho (x_0))$, with $\rho (x_0) \neq 1$.  
   Lemma~\ref{lem:analysis_fcts} implies $\rho (x_0) = y_0$.
\end{proof}

\begin{lem} \label{lem:cauchy_rho_0}
    There exists a unique $\rho_c \in \mathcal{C}^2$ satisfying
    \eqref{eq:rho_ode}, $\rho_c (0) = y_0$ and $\rho_c' (0) =
    0$. Moreover, $\rho_c \geq y_0 > 0$, $\lim_{\pm \infty} \rho_c =
    1$ and $\rho_c - 1 \in H^1$. 
\end{lem}

\begin{proof}
 Since $f_c$ is $\mathcal{C}^\infty$ on $(0, \infty)$, we can apply
Cauchy-Lipschitz Theorem to get a local solution. By multiplying
\eqref{eq:rho_ode} by $\rho_c'$ and integrating, this solution
 satisfies also \eqref{eq:energy_ode_c_compl}. 
Lemma~\ref{lem:analysis_fcts} then  yields $\rho_c\ge y_0 $,  
which proves that this solution is global.

 Then, by \eqref{eq:rho_ode} and Lemma \ref{lem:analysis_fcts}
    again, we get $\rho_c'' (0) > 0$. Hence there exists $\eps > 0$
    such that $\rho_c' < 0$ on $(- \varepsilon, 0)$ and $\rho_c' > 0$
    on $(0, \varepsilon)$, and $\rho_c > y_0$ on $(- \varepsilon,
    \varepsilon) \setminus \{ 0 \}$. 
    Let
    \begin{align*}
        x^+ &\coloneqq \sup \{ x > 0 \, | \, \rho_c' (y) > 0 \text{ for all } y \in (0, x) \}, \\
        x^- &\coloneqq \inf \{ x < 0 \, | \, \rho_c' (y) < 0 \text{ for all } y \in (x, 0) \}.
    \end{align*}
    From \eqref{eq:energy_ode_c_compl} and
    Lemma~\ref{lem:analysis_fcts} along with the fact that $\rho_c$ and
    $\rho_c'$ are continuous, we know that either $x^+ = + \infty$ or
    $\rho_c (x^+) = 1$, and similarly either $x^- = - \infty$ or $\rho_c
    (x^-) = 1$. However, if $\rho (x_1) = 1$ at some point $x_1 \in
    \mathbb{R}$, then \eqref{eq:energy_ode_c_compl} gives $\rho_c' (x_1)
    = 0$ and the uniqueness in Cauchy-Lipschitz Theorem for \eqref{eq:rho_ode}
    leads to $\rho_c \equiv 1$, hence a contradiction. Therefore, $x^+ =
    + \infty$ and $x^- = - \infty$, which implies that $\rho_c$ is
    decreasing on $(- \infty, 0)$, increasing on $(0, \infty)$ and
    $\rho_c (x) \in [y_0, 1)$ for all $x \in \mathbb{R}$. Thus, it has
    limits $\ell_\pm > y_0$ at $\pm \infty$ and must also satisfy
    $\lim_{\pm \infty} \rho_c' = 0$ in view of
    \eqref{eq:energy_ode_c_compl}. Therefore, the limit must satisfy 
    $g_c (\ell_\pm) = 0$, which leads to $\ell_\pm = 1$. 

    Last, since $\rho_c' > 0$ on $(0, \infty)$,
    \eqref{eq:energy_ode_c_compl} reads on this interval 
    \begin{equation*}
        \rho_c' = \sqrt{g_c (\rho_c)}.
    \end{equation*}
    By Lemma \ref{lem:analysis_fcts} and since $\rho_c (x) \geq y_1$ for $x$ large enough, we get
    \begin{equation*}
        \rho_c' \geq \sqrt{C_c} (1 - \rho_c).
    \end{equation*}
    Thus $\rho_c$ converges exponentially fast  to $1$ at $+ \infty$ by
    Gronwall lemma, and therefore  $\rho_c'$ converges exponentially
    fast to $0$. The same
     holds as $x \rightarrow - \infty$, hence
    the conclusion $\rho_c - 1 \in H^1$. 
\end{proof}

\begin{lem} \label{lem:exist_trav_waves_c}
    For  $c$ such that $0 < c^2 < 2 \lambda$ and  $\theta_0 \in
    \mathbb{R}$, define $\Theta_c$ by 
    \begin{equation*}
        \Theta_c' = \frac{c}{2} \Bigl( 1 - \frac{1}{\rho_c^2} \Bigr), \qquad
        \Theta_c (0) = \theta_0,
    \end{equation*}
    where $\rho_c$ is given by Lemma~\ref{lem:cauchy_rho_0}. 
    Then $\phi_c \coloneqq \rho_c e^{i \Theta_c} \in
    E_\textnormal{logGP}$ and $u$ defined by \eqref{eq:trav_wave} is a
    traveling wave for \eqref{logGP}. 
\end{lem}

\begin{proof}
    Direct computations show that $\Theta_c$ satisfies
    \eqref{eq:ode_theta}, and thus $(\rho_c,\Theta_c)$ satisfies
    \eqref{eq:trav_ode_theta_rho}, which is equivalent to $\phi_c$
    satisfying \eqref{eq:ode_phi_trav} or $u$ satisfying
    \eqref{logGP}. Therefore, we only have to prove that $\phi_c \in
    E_\textnormal{logGP}$. 

    By Lemma~\ref{lem:cauchy_rho_0}, we already know that $\abs{\phi_c} - 1 = \rho_c - 1 \in H^1$. Moreover,
    \begin{equation*}
        \phi_c' = \( \rho_c' + i \rho_c \Theta_c'\) e^{i \Theta_c} = \( \rho_c' + i \frac{c}{2} \frac{\rho_c^2 - 1}{\rho_c} \) e^{i \Theta_c}.
    \end{equation*}
    By Lemma \ref{lem:cauchy_rho_0} again, $\rho_c' \in L^2$ and
    $\rho_c$ is bounded, and bounded away from $0$, so that
    $\frac{\rho_c^2 - 1}{\rho_c} \in L^2$. We conclude  $\phi_c \in
    E_\textnormal{logGP}$ thanks to Lemma~\ref{lem:E_logGP_descr}. 
\end{proof}

\begin{proof}[End of the proof of Theorem~\ref{th:sol_wave}]
    Let $\phi \in E_\textnormal{logGP}$ and $u$ a traveling wave of
    the form \eqref{eq:trav_wave}. By applying Lemma
    \ref{lem:carac_trav_waves}, we get $x_0 \in \mathbb{R}$ such that
    $\rho (x_0) = y_0$ where $\rho = \abs{\phi}$, which yields $\rho'
    (x_0) = 0$ through \eqref{eq:energy_ode_c_compl}. Since $\rho (. +
    x_0)$ also satisfies \eqref{eq:rho_ode}, this function satisfies
    the assumptions of Lemma \ref{lem:cauchy_rho_0}. Thus $\rho (. +
    x_0) = \rho_c$, and since $\theta$ satisfies \eqref{eq:theta_ode},
    $\phi$ is of the form $\phi_c$ defined in Lemma
    \ref{lem:exist_trav_waves_c} for some $\theta_0$. 
\end{proof}


\section{Some open questions}
\label{sec:open}

\subsection*{Multidimensional solitary and traveling waves}

We have classified solitary and traveling waves in the one-dimensional
setting, using ODE techniques. The picture is of course rather
different for $d\ge 2$, see typically the case of traveling waves for
the Gross-Pitaevskii equation \eqref{eq:GP}, in
\cite{BeSa99,BeGrSa09}, where the role of hydrodynamical formulation
(via Madelung transform) is decisive.  See also \cite{BeGrSa08,ChMa17} and references
therein. 
\smallbreak

We note that contrary to the case with vanishing boundary condition at
infinity, one
cannot easily construct multidimensional waves from one-dimensional
ones by  using the tensorization property evoked in the introduction. Recall that if 
$u_1(t,x_1),\dots ,u_d(t,x_d)$ are solutions to the one-dimensional
logarithmic Schr\"odinger equation, then
$u(t,x):=u_1(t,x_1)\times\dots\times u_d(t,x_d)$ solves 
the $d$-dimensional logarithmic Schr\"odinger equation. In the case of 
 vanishing boundary condition at
infinity, $u_j(t,\cdot)\in L^2(\R)$, hence $u(t,\cdot)\in
L^2(\R^d)$; see e.g. \cite{CaFe21} for some applications. On the other
hand, if $u_j(t,\cdot)\in 
E_\textnormal{logGP}=E_\textnormal{GP}$ (the sets are
the same in the one-dimensional case), then even for $d=2$, 
\begin{equation*}
  (x_1,x_2)\mapsto u_1(t,x_1)u_2(t,x_2)
\end{equation*}
does not belong to the energy space
$E_\textnormal{logGP}=E_\textnormal{GP}$ (case of $\R^2$), as the
spatial derivatives are not in $L^2(\R^2)$.

\subsection*{Orbital stability}

In the case of the one-dimensional Gross-Pitaevskii equation, the
orbital stability of traveling waves was proven in \cite{GeZh09} by
using the fact that this equation is integrable, in the sense that it
admits a Lax pair, and solutions are analyzed thanks to
Zakharov-Shabat's inverse scattering method. It seems unlikely that
even for $d=1$, \eqref{logGP} is completely integrable, so the
stability of traveling waves will require a different
approach. 
\smallbreak

In \cite{Lin02,BGSS08}, the orbital stability of the black soliton is
studied without using the integrable structure, but variational
arguments. 
In \cite{BGS15} (dark soliton) and \cite{GrSm15} (black soliton), the
orbital stability is improved to asymptotic 
stability (the modulation factors, which consist of a translation
parameter and a phase shift, are analyzed). These approaches may be
more suited to 
\eqref{logGP}, but new 
difficulties due to the special form of the nonlinearity arise,
starting with the control of the presence of vacuum, $\{u=0\}$.

\subsection*{Scattering}

Another way to study stability properties consists in establishing a
scattering theory, considering the plane wave solution, $u(t,x) =
e^{ik\cdot x-i|k|^2t}$, $k\in \R^d$, as a reference solution. By
Galilean invariance, the analysis is reduced to the case $k=0$, and
the question is the behavior of the solution as $t\to +\infty$ when
the initial datum $u_0\equiv 1$ is perturbed. In such a framework, a
crucial object is the linearized Schr\"odinger operator about the
constant $1$. Using in addition normal form techniques, a scattering
theory for \eqref{eq:GP} was developed in
\cite{GuNaTs06,GuNaTs07,GuNaTs09} for $d\ge 2$, and resumed in
\cite{KiMuVi16,KiMuVi18} in the case of the three-dimensional
cubic-quintic Schr\"odinger equation (the quintic term introduces the
new difficulty of an energy-critical factor). 
\smallbreak

In the case of \eqref{logGP}, the new difficulty is due to the
fact that the logarithmic nonlinearity is not multilinear in $(u,\bar
u)$, reminding of the difficulty of giving sense to linearization in
such a context (see e.g. \cite{Fe21,Pel17}). 

\subsection*{Convergence toward other models}

It is well-established that in the long wave r\'egime, the
Gross-Pitaevskii equation converges to the KdV equation in the
one-dimensional case, and to the KP-I equation in the multidimensional
case; see e.g.  \cite{BGSS09,BGSS10,ChRo10}.  More precisely, we may
resume the general presentation from \cite{ChRo10}: considering the
nonlinear Schr\"odinger equation
\begin{equation*} i\d_t u + \Delta u = f\(|u|^2\)u,
\end{equation*} where the nonlinearity $f$ is such that $f(1)=0$ and
$f'(1)>0$,  the KdV and KP-I equations appear in the limit $\eps\to
0$, after the rescaling
\begin{equation*} t= c\eps^3\tau,\quad X_1=\eps(x_1-ct),\quad
X_j=\eps^2 x_j, \ j\in \{2,\dots,d\}. 
\end{equation*} Writing 
\begin{equation*} u(t,X) = \(1+\eps^2
A^\eps(t,X)\)e^{i\eps\varphi^\eps(t,X)},
\end{equation*} and plugging this expression into the equation solved
by $u$, the formal limit $\eps\to 0$ in the obtained system provides a
limit $A$ solution to the KdV equation if $d=1$, and to the KP-I
equation if $d\ge 2$. We emphasize that in the usual Gross-Pitaevskii
equation \eqref{eq:GP}, $f(y)=y-1$, in the logarithmic case
\eqref{logGP} with $\lambda=1$,  $f(y)=\ln y$, so we always have
$f(1)=0$ and $f'(1)>0$.  The general case considered in \cite{ChRo10}
is analyzed under the assumption $f\in \mathcal C^\infty (\R,\R)$,
which is not satisfied in the logarithmic case. This is probably not
the sharp assumption in order  for the arguments presented in
\cite{ChRo10} to remain valid, but it is most likely that the proof
uses a $\mathcal C^0$ regularity on $[0,+\infty)$, like in
\cite{Ch12,ChMa17,Ma13} (these papers are devoted to the  existence of
traveling waves in the same framework; $\mathcal C^1$
regularity of $f$ is required only near $1$). Again, the main difficulty in
the case of the logarithmic nonlinearity lies in the control of
vacuum, as $f$ fails to be continuous at the origin.


\bibliographystyle{abbrv}
\bibliography{biblio}
\end{document}